\documentclass[11pt]{amsart}
\usepackage{amsmath}
\usepackage{amssymb}
\usepackage{graphicx}

\newtheorem{theorem}{Theorem}[section]
\newtheorem{corollary}[theorem]{Corollary}
\newtheorem{lemma}[theorem]{Lemma}
\newtheorem{proposition}[theorem]{Proposition}
\theoremstyle{definition}
\newtheorem{definition}[theorem]{Definition}

\newtheorem{remark}[theorem]{Remark}

\usepackage{mathrsfs,amsfonts}
\usepackage{amssymb,amsmath,amsthm,amsfonts}

\newcommand{\mm}     {{\hbox{\hskip 0.5pt}}}
\newcommand{\m}      {{\hbox{\hskip 1pt}}}
\newcommand{\nm}     {{\hbox{\hskip -3pt}}}
\newcommand{\bluff}  {{\hbox{\raise 15pt \hbox{\mm}}}}
\newcommand{\sbluff} {{\hbox{\raise  9pt \hbox{\mm}}}}
\newfont{\roma}{cmr10 scaled 1200}

\def\dref#1{(\ref{#1})}
\def\crr{\cr\noalign{\vskip2mm}}
\def\disp{\displaystyle}

\renewcommand{\L}    {{\Lambda}}
\renewcommand{\l}    {{\lambda}}
\renewcommand{\o}    {{\omega}}
\newcommand{\Om}     {{\Omega}}

\newcommand{\half}   {{\frac{1}{2}}}

\newcommand{\dd}     {{\rm d\hbox{\hskip 0.5pt}}}
\newcommand{\rarrow} {\mathop{\rightarrow}}
\newcommand{\FORALL} {{\hbox{$\hskip 11mm \forall \;$}}}
\newcommand{\CL}     {{C_\Lambda}}
\renewcommand{\Re}   {{\rm Re\,}}
\renewcommand{\Im}   {{\rm Im\,}}

\newcommand{\aline}  {{\mathbb A}}
\renewcommand{\cline}{{\mathbb C}}
\renewcommand{\hline}{{\mathbb H}}
\newcommand{\nline}  {{\mathbb N}}
\newcommand{\rline}  {{\mathbb R}}
\newcommand{\tline}  {{\mathbb T}}
\newcommand{\zline}  {{\mathbb Z}}
\newcommand{\Fscr}   {{\mathcal F}}
\newcommand{\Kscr}   {{\mathcal K}}
\newcommand{\Lscr}   {{\mathcal L}}
\newcommand{\Oscr}   {{\mathcal O}}
\newcommand{\AAA}    {{\bf A}}
\newcommand{\GGG}    {{\bf G}}
\newcommand{\Kf}     {{K_{\hbox{\hskip -0.5mm}f}}}

\newcommand{\bbm}[1]{\left[\begin{matrix} #1 \end{matrix}\right]}
\newcommand{\sbm}[1]{\left[\begin{smallmatrix} #1
   \end{smallmatrix}\right]}

\numberwithin{equation}{section}

\title[Regulator problem for Schr\"{o}dinger equation]
      {Solving the regulator problem for the one-dimensional
       Schr\"{o}dinger equation via backstepping}

\author[H.C. Zhou]{Hua-Cheng Zhou}
\address[H.C. Zhou]{School of Mathematics and Statistics, Central 
         South University, Changsha, 410075, PR China}
\email{{\tt hczhou\m @\m amss.ac.cn}}

\author[G. Weiss]{George Weiss}
\address[G. Weiss]{Corresponding author. School of Electrical 
         Engineering, Tel Aviv University, Ramat Aviv 69978, Israel}
\email{\tt gweiss\m @\m eng.tau.ac.il}

\keywords{regulator problem, Schr\"{o}dinger
equation, backstepping, controller with internal loop, exosystem,
admissible observation operator, compatible system node, observer.
\\ \m \ \ The first author is supported by the National Natural 
Science Foundation of China (grant no. 61803386). The second author 
is a partner (coordinator) in the ITN project ConFlex. This project
is funded by the European Union's Horizon 2020 research and 
innovation programme under the Marie Sklodowska-Curie grant 
agreement no. 765579.}


\begin{document}
\begin{abstract}
We investigate the regulator problem (tracking and disturbance
rejection) for a system (plant) described by a boundary controlled
anti-stable linear one-dimensional Schr\"{o}dinger equation, using the
backstepping approach. The output to be controlled is not required to
be measurable and its observation operator is assumed to be admissible
for a certain operator semigroup that is related to the operator
semigroup of the original plant. We consider both the state feedback
and the output feedback regulator problem. In the latter case, the
measurement from the Schr\"odinger equation is taken at the boundary.
First we show that the open-loop system is well-posed. We design a
state feedback control law that solves the regulator problem by the
backstepping method. Then, a finite-dimensional reference observer and
an infinite-dimensional disturbance observer are designed.  Putting
these together, we obtain an output feedback controller with internal
loop that achieves output regulation.
\end{abstract}

\maketitle

\section{Introduction and problem formulation} 

\ \ \ The {\em regulator problem} is one of the fundamental issues in
control theory. It concerns tracking a reference signal $r$ with a
certain output $y$ of the plant, while rejecting a disturbance signal
$d$, where both $r$ and $d$ are generated by a marginally stable
finite-dimensional exosystem. In the {\em state feedback regulator
problem}, the controller has access to the state of the exosystem and
also to the state of the plant. In the {\em output feedback regulator
problem}, the controller has access to a measurement output $y_m$
(that may be different from $y$) and also to $r$. (This is related to
the more often encountered {\em error feedback regulator problem},
where the signal available to the controller is $e_y=r-y$.) It is also
required that the closed-loop system (not including the exosystem)
should be stable, in some suitable sense (e.g., exponentially).

The main approach to the output (or error) feedback regulator problem
is the {\em internal model principle} \cite{Davison1976tac,Fran:1976}.
Research in this branch of control theory has been active for over 30
years \cite{AsOr:03,Byrnes_2000, ImmPoh_2006,Isi:95,IsBy:90,JaWe:08,
KIF,PauPoh_2014,Viv2014tac,ReWe_2003}. The first results concerning
the regulator problem were developed for lumped parameter linear
systems, see \cite{Davison1976tac,Fran:1976}. These results were
extended to distributed parameter systems in \cite{Byrnes_2000}, where
the control and observation operators are bounded, and in
\cite{Viv2014tac,PauPoh_2014,ReWe_2003}, where the control
and observation operators are unbounded but admissible. In all these
references, the exosystem is assumed to be finite-dimensional, while
in \cite{Immonen2005tac,ImmPoh_2006,Paun_2017}, it is
infinite-dimensional. Another powerful method in dealing with the
regulator problem is the backstepping approach. In \cite{Deu2015auto},
the regulator problem for a boundary controlled parabolic PDEs is
solved using the backstepping approach. This method is again used for
the robust output regulation of parabolic PDEs in
\cite{Deu2015tac}. An interesting recent work is \cite{LBscl2015},
where based on backstepping, the output tracking problem is considered
for a general $2\times2$ system of first order linear hyperbolic PDEs,
but no disturbances are taken into consideration. Adaptive control is
used for output tracking for the Schr\"{o}dinger equation in
\cite{Liu_2018}, where the system is exponentially stable and the
disturbance acts at the boundary. For the optimal regularity, sharp
uniform decay rates and observability of Schr\"{o}dinger equations
in several space dimensions, we refer to the work of Irena Lasiecka
and collaborators \cite{LT1992DIE,LT2006JEE,LTZ2004I,LTZ2004II}.

We consider the following one-dimensional Schr\"{o}dinger equation
with Neumann boundary control and both distributed and boundary
disturbance, with $t\geq 0$: \vspace{-3mm}
\begin{equation} \label{Sch}
  \left\{\begin{array}{l} z_t(x,t) \m=\m -iz_{xx}(x,t) + h(x)z(x,t)
  + g(x)d_1(t),\quad 0<x<1,\crr\disp z_x(0,t) \m=\m -iqz(0,t) + d_2
  (t),\qquad z_x(1,t) \m=\m u(t),\crr\disp z(x,0) \m=\m z_0(x),\quad
  0\leq x\leq 1,\crr\disp y(t)\m=\m C_e[z(\cdot,t)],\qquad y_m(t)
  \m=\m z(1,t).\end{array}\right.
\end{equation}
We denote by $z'(x,t)$ or $z_x(x,t)$ the derivative of $z(x,t)$ with
respect to $x$ and by $\dot{z}(x,t)$ or $z_t(x,t)$ the derivative of
$z(x,t)$ with respect to $t$. $u(t)$ is the control input signal, $y$
is the output signal to be controlled, $y_m$ is the measurement (the
information available to the controller), $d_1(t),d_2(t)$ are the
disturbances, $z_0$ is the initial state, $q>0$ and $h,g\in C[0,1]$
are known. The system \dref{Sch} is a typical unmatched boundary
control problem: the control $u$ acts on one end of the domain and one
disturbance $d_2$ acts on the other end (while the other disturbance
$d_1$ acts distributed).

We consider the system \dref{Sch} in the energy state space
$\hline=L^2[0,1]$ with the usual inner product and norm. We will also
use the Sobolev spaces $H^1(0,1)$ and $H^2(0,1)$, with their usual
norms. If $z\in C([0,\infty),\hline)$, then instead of $[z(t)] (x)$ we
write $z(x,t)$. The observation operator $C_e$ in \dref{Sch} is a
bounded linear functional on $H^2(0,1)$ (not specified). We call $C_e$
{\em bounded} if it has a continuous extension to $\hline$ and {\em
unbounded} otherwise.

We will often need to refer to the {\em unperturbed system} (perhaps
not the best name) that is obtained from \dref{Sch} by setting $d_1(t)
=0$ (for all $t\geq 0$), and also $h(x)=0$ (for all $x\in[0,1]$):
\begin{equation} \label{Sch_unp}
  \left\{\begin{array}{l} z_t(x,t) \m=\m -iz_{xx}(x,t)
  \quad 0<x<1,\crr\disp z_x(0,t) \m=\m -iqz(0,t)+d_2(t),\qquad
  z_x(1,t) \m=\m u(t),\crr\disp z(x,0) \m=\m z_0(x),\quad
  0\leq x\leq 1,\crr\disp y(t)\m=\m C_e[z(\cdot,t)],\qquad y_m(t)
  \m=\m z(1,t).\end{array}\right.
\end{equation}

We introduce the operator $A$ as the generator of the operator
semigroup $\tline$ that describes the evolution of the state
$z(\cdot,t)$ of \dref{Sch_unp} in $\hline$ if the inputs are
$d_2=0$ and $u=0$: \vspace{-1mm}
\begin{equation} \label{A-def}
   Af \m=\m -if'',\quad D(A) \m=\m \{f\in H^2(0,1)\ |\
   f'(0) \m=\m -iqf(0),\ f'(1)=0\} \m.
\end{equation}
We shall investigate this semigroup in Lemma \ref{lem-A}. {\bf We
assume that $C_e$ (restricted to $D(A)$) is an admissible observation
operator for the operator semigroup $\tline$ generated by $A$.} The
concept of admissible observation operator will be recalled at the
beginning of Sect.~2.

For instance, the above assumption is true if $C_e$ is the sum of a
point observation operator and a distributed observation operator,
which means that \vspace{-3mm}
\begin{equation} \label{Cop-def}
   y(t) \m=\m C_e[z(\cdot,t)] \m=\m \theta z(x_0) + \int_0^1 c(x)
        z(x,t)\dd x,
\end{equation}
where $\theta\in\cline$, $x_0\in[0,1]$ and $c\in L^2[0,1]$
(the proof of this is similar to the proof of Lemma \ref{lem-A}).

A triple $(z,\sbm{d_2\\ u},y)$ is called a {\em classical
solution} of \dref{Sch_unp} on $[0,\infty)$ if:

  (a) $z\in C^1([0,\infty);\hline)$,

  (b) $d_2,u,y\in C[0,\infty)$,

  (c) $z(t)\in H^2(0,1)$ holds for all \m $t\geq 0$,

  (d) \dref{Sch_unp} holds for all \m $t\geq 0$.

The system \dref{Sch_unp} has many classical solutions. Indeed, we
show in Proposition \ref{Brighitte} that if $d_2,u\in H^1_{loc}(0,
\infty)$ and $z_0\in H^2(0,1)$ are such that $z_0'(0)=-iqz_0(0)+d_2
(0)$ and $z_0'(1)= u(0)$, then \dref{Sch_unp} has a corresponding
classical solution on $[0,\infty)$. A similar statent holds for
\dref{Sch}, see Corollary \ref{ProAB}. Moreover, the systems
\dref{Sch} and \dref{Sch_unp} are well-posed, see Proposition
\ref{keso_van}.

We suppose, as is common in regulator theory, that there exists a
linear system with no input, referred to as the {\em exosystem}
(sometimes called the exogenous system), that generates both the
disturbances $d_1,d_2$ and the reference $r$ (these are all scalar
signals): \vspace{-1mm}
\begin{equation} \label{dis-produce}
   \begin{array}{l} \dot{w}(t) \m=\m Sw(t),\ \ t>0,\;\; w(0)=w_0\in
   \rline^{n_w},\crr d_1(t) \m=\m p^{\top}_1 w(t)=q_{d_1}^{\top}w_d
   (t),\ \ t\geq 0,\crr d_2(t) \m=\m p^{\top}_2w(t) \m=\m q_{d_2}
   ^{\top}w_d(t),\ \ t\geq 0,\crr r(t) \m=\m p^{\top}_rw(t) \m=\m
   q_{r}^{\top}w_r(t),\ \ t\geq 0. \end{array}
\end{equation}
Here, $S$ is a block diagonal matrix $S={\rm diag}(S_d,S_r)$, which
leads with $w=\sbm{w_d\\ w_r}$ to the signal models
$\dot{w}_d=S_dw_d$, $w_d(0)=w_{d_0}\in\cline^{n_d}$, and
$\dot{w}_r=S_rw_r$, $w_r(0)=w_{r_0}\in\cline^{n_r}$, $n_d+n_r=n_w$.
Clearly $q_{d_1},q_{d_2}\in\cline^{n_d}$. {\bf We assume that $S$ is a
diagonalizable matrix, all its eigenvalues are on the imaginary axis,
the eigenvalues of $S_d$ are distinct and $(q^{\top}_r,S_r)$ is
observable.} The disturbances cannot be measured and the reference
signal is available to the controller.

{\bf Our objective} is to design an output feedback regulator such
that for all initial states of the systems \dref{Sch} and
\dref{dis-produce}, the following requirements are satisfied: (i) All
the internal signals are bounded. (ii) If the observation operator
$C_e$ is bounded, then we design a state feedback control law, using
the state $z(\cdot,t)$ of \dref{Sch} as well as the state $w(t)$ of
\dref{dis-produce}, such that the tracking error $e_y=y-r$ is
exponentially vanishing: there exist constants $m_0,\mu_0>0$ such
that \vspace{-1mm}
\begin{equation} \label{ey-to0}
   |e_y(t)| \m\leq\m m_0 e^{-\mu_0 t} \FORALL t\geq 0 \m.
\end{equation}
Based on this, we also design an output feedback controller, a
dynamical system with inputs $y_m(t)$ and $r(t)$, such that in the
closed-loop system, \dref{ey-to0} holds.

Alternatively, if $C_e$ is unbounded but admissible, then we design a
state feedback controller and an output feedback controller (with
internal loop), such that for some $\alpha<0$, \vspace{-2mm}
\begin{equation} \label{ourerr-def}
   e_y \m\in\m L_{\alpha}[0,\infty),
\end{equation}
where $L_{\alpha}[0,\infty)$ is a weighted function space defined by
\vspace{-1mm}
$$ L_{\alpha}^2[0,\infty) :=\m \bigg\{f\in L^2_{loc}[0,\infty) \
   \bigg| \int_0^\infty e^{-2\alpha t}|f(t)|^2 \dd t \m<\m \infty
   \bigg\}.$$
For the concept of stabilizing controller with internal loop we refer
to \cite{WeCu_97,CWW_2001}. Essentially it means that the controller
is well-posed and to create the well-posed and stable closed-loop
system, we have to close two feedback loops: one involving the plant
and the controller and another one (called the internal loop)
involving the controller only. Closing the internal loop on the
controller only (without the plant) may lead to a non-well-posed
system.

The outline of the paper is as follows: In Sect.~2 we give a bit of
mathematical background on compatible system nodes, admissibility and
well-posedness. In Sect.~3 we derive various properties of the
Schr\"odinger equation system \dref{Sch}, which we reformulate in the
operator theoretic language. In Sect.~4 we solve the state feedback
regulator problem, while using backstepping for the stabilization.
Sect.~5 is devoted to the design of an observer for the combined
system \dref{Sch} and \dref{dis-produce}, using again a backstepping
transformation. In Sect.~6, based on the estimated state from the
observer, we show how to solve the output feedback regulator problem.

\section{Some background on well-posed system nodes} 

\ \ \ In this section we recall some general facts on admissible
control and observation operators, compatible system nodes, classical
solutions, well-posedness, transfer functions, feedback and
closed-loop systems, following \cite{StWe_12}, \cite{obs_book},
\cite{TuWe_14_survey} and \cite{art12}. For a better understanding of
these topics and for the proofs, the reader is advised to look up the
mentioned references.

Let $X,U$ and $Y$ be Hilbert spaces, let $\tline_t$ be a strongly
continuous semigroup of operators on $X$ with generator $A$, let
$X_{1}$ be the space $D(A)$ with the norm $\|x\|_{1}=\|(\beta I-A)x
\|$ and let $X_{-1}$ be the completion of $X$ with respect to the norm
$\|x\|_{-1}=\|(\beta I-A)^{-1}x\|$, where $\beta$ is an arbitrary (but
fixed) element in the resolvent set $\rho(A)$. An operator
$B\in\Lscr(U,X_{-1})$ is called an {\em admissible control operator}
for $\tline$ if for some (hence, for every) $\tau>0$ and for every
$u\in L^2([0,\infty);U)$,
$$ \int_0^{\tau} \tline_{\tau-s} Bu(s) \dd s\in X.$$
In this case, for any $x_0\in X$ and any $u\in L^2_{loc}([0,\infty);
U)$ the equation $\dot x=Ax+Bu$ has a unique solution in $X_{-1}$
that satisfies $x(0)=x_0$, and moreover we have $x\in
C([0,\infty);X)$.

An operator $C\in\Lscr(X_1,Y)$ is an {\em admissible observation
operator} for $\tline$ if for some (hence, for every) $\tau>0$ there
exists $m_\tau>0$ such that
$$ \int_0^{\tau} \|C\tline_t x\|^2\dd t \m\leq\m m_{\tau}\|x\|^2
   \FORALL x\in D(A).$$
The $\L$-extension of an operator $C\in\Lscr(X_1,Y)$ (with respect
to $A$), denoted $\CL$, is defined as follows:
\begin{equation} \label{MillionDollars}
  \CL x \m=\m \lim_{\l\to\infty} C\l(\l I-A)^{-1} x
\end{equation}
and its domain $D(\CL)$ consists of those $x\in X$ for which the above
limit exists. In this case, by \cite[Proposition 4.3.6]{obs_book},
for every $x\in X$, the output $y(t)=\CL\tline_t x$ exists for
almost every $t\geq 0$ and
\begin{equation} \label{Eilat}
   y(t)=\CL\tline_t x \Rightarrow y\in L^2_{\alpha}([0,\infty);Y) \
   \mbox{ for all }\ \alpha>\o_{\tline} \m,
\end{equation}
where $\o_{\tline}$ is the growth bound of the semigroup $\tline$. We
have that $C$ is an admissible observation operator for $\tline$ if and
only if $C^*$ is an admissible control operator for $\tline^*$.

Let $U,X,Y$ and $A$ be as above, and let $B\in\Lscr(U,X_{-1})$.
We introduce the space \vspace{-2mm}
$$D(S)\m=\m\left\{\sbm{x\\ u}\in X\times U\ |\ Ax+Bu\in X\right\}\m.$$
We also define the space $Z\subset X$ that consists of all the vectors
$z\in X$ that can be the first component of a vector in $D(S)$:
\vspace{-1mm}
\begin{equation} \label{yellow_vests}
   Z \m=\m D(A) + (\beta I-A)^{-1} BU \m,
\end{equation}
which is independent of the choice of $\beta\in\rho(A)$. This is a
Hilbert space with the norm \vspace{-1mm}
$$ \|z\|_Z^2 \m=\m \inf\left\{ \|x\|_1^2 + \|v\|^2 \ |\ \text{$x\in
   X_1$, $v \in U$, $z = x + (\beta I-A)^{-1} Bv$} \right\} \m.$$

Let $C:D(C)\rarrow Y$ be such that $Z\subset D(C)$ and the restriction
of $C$ to $Z$ is in $\Lscr(Z,Y)$. Finally, let $D\in\Lscr(U,Y)$. Then
$(A,B,C,D)$ is called a {\em compatible system node} on $(U,X,Y)$. (We
mention that we took a short-cut here: in the cited references, and
several others, the more general and complicated concept of system
node is introduced first, and compatible system nodes are introduced
later as a special case. It is easy to show that our definition above
is equivalent to the one in \cite{StWe_12, TuWe_14_survey}. In the
cited references, the notation $\overline C$ appears instead of $C$,
and $C$ is $\overline C$ restricted to $D(A)$.)

To a compatible system node as above we associate its {\em system
operator} $S:D(S)\rarrow X\times Y$: \vspace{-2mm}
$$S \m=\m \bbm{A & B \\ C & D} \m.$$
The compatible system node is usually associated with the equation
\begin{equation} \label{node_evol}
  \bbm{\dot x(t)\\ y(t)} \m=\m S\bbm{x(t)\\ u(t)} \FORALL t\geq 0 \m,
\end{equation}
where $u,x$ and $y$ have the meaning of input, state and output
functions. $B$ and $C$ are called the {\em control operator} and the
{\em observation operator} of the system node, respectively.

In the spirit of \cite[Sect.~3]{StWe_12}, \cite[Sect.~4]
{TuWe_14_survey}, we define the following concept:

\begin{definition} \label{ClassSolDef} {\rm
Let $S$ be the system operator of a compatible system node $(A,B,C,
D)$ on $(U,X,Y)$. A triple $(x,u,y)$ is called a {\em classical
solution} of \dref{node_evol} on $[0,\infty)$ if:

  (a) $x\in C^1([0,\infty);X)$,

  (b) $u\in C([0,\infty);U)$, \ $y\in C([0,\infty);Y)$,

  (c) $\sbm{x(t)\\ u(t)}\in D(S)$ for all $t\geq 0$,

  (d) \dref{node_evol} holds.}
\end{definition}

\begin{proposition} \label{smooth_comp}
With the notation of the last definition, if \m $u\in C^2([0,
\infty);U)$ and $\sbm{x_0\\ u(0)}\in D(S)$\m, then the equation
{\rm\dref{node_evol}} has a unique classical solution $(x,u,y)$
satisfying $x(0)=x_0$.
\end{proposition}

For the proof we refer to Proposition 4.2.11 in \cite{obs_book} (it
also appears in various other references). Under the conditions of the
above proposition, we have \vspace{-1mm}
$$ x(t) \m=\m \tline_t\m x(0) + \int_0^t \tline_{t-\sigma}Bu(\sigma)
   \dd \sigma \FORALL t\geq 0 \m.$$

\begin{definition} \label{WellPosedDef} {\rm
With the notation of the previous definition, $(A,B,C,D)$ is {\em
well-posed} if for some (hence, for every) $\tau>0$ there is a
$K_\tau>0$ such that for every classical solution $(x,u,y)$ of
\dref{node_evol}, \vspace{-2mm}
$$ \|z(\tau)\|^2+\int_0^\tau|y(t)|^2 \dd t \m\leq\m K_\tau \left(
   \|z(0)\|^2+\int_0^\tau|u(t)|^2 \dd t \right) \m.$$
Here, $\tline$ is the operator semigroup generated by $A$.}
\end{definition}

We will use the term ``well-posed system node'' instead of the
cumbersome ``well-posed compatible system node''. There is a good
justification for this, see \cite[Proposition 4.5]{TuWe_14_survey}.

\begin{proposition} \label{Macron}
We use the notation of Definition {\rm\ref{ClassSolDef}} and we
denote again by $\tline$ the operator semigroup generated by $A$.

If $(A,B,C,D)$ is well-posed, then it follows that $B$ is an
admissible control operator for $\tline$, $C$ (restricted to $D(A)$)
is an admissible observation operator for $\tline$, and the
{\em transfer function} of $(A,B,C,D)$, defined by \vspace{-1mm}
\begin{equation} \label{Mattis_resigns}
   \GGG(s) \m=\m C(sI-A)^{-1}B+D \FORALL s\in\cline\
   \mbox{ with }\ \Re s>\o_\tline \m,
\end{equation}
is bounded on any half-plane $\cline_\alpha=\{s\in\cline\ |\ \Re s>
\alpha\}$, if $\alpha>\o_\tline$.

Conversely, if $B$ is an admissible control operator for $\tline$,
$C$ is an admissible observation operator for $\tline$ and $\GGG$ is
bounded on some right half-plane, then it follows that $(A,B,C,D)$
is well-posed.
\end{proposition}

The following simple perturbation result will be useful.

\begin{proposition} \label{Trump}
Let $(A,B,C,D)$ be a well-posed system node on $(U,X,Y)$ and let
$P\in\Lscr(X)$. Then $(A+P,B,C,D)$ is again a well-posed system node
on $(U,X,Y)$.
\end{proposition}

\begin{proof}
Assume that $(A,B,C,D)$ is well-posed, hence (according to the
previous proposition) $A$ and $B$ are admissible for $\tline$, the
semigroup generated by $A$. It follows from \cite[Theorem 5.4.2 and
Corollary 5.5.1]{obs_book} that $B$ and $C$ are admissible also for
the semigroup generated by $A+P$, and the spaces $X_1$ and $X_{-1}$
remain the same for $A+P$. According to Proposition \ref{Macron}, the
transfer function $\GGG$ from \dref{Mattis_resigns} is bounded on some
right half-plane.  Denoting the transfer function of the compatible
system node $(A+P,B,C,D)$ by $\GGG_P$, we have the elementary identity
\vspace{-1mm}
\begin{equation} \label{Gad_Hareli}
   \GGG_P(s)-\GGG(s) \m=\m C(sI-A)^{-1}P(sI-A-P)^{-1}B \m.
\end{equation}
The functions $C(sI-A)^{-1}$ and $(sI-A-P)^{-1}B$ are bounded on some
right half-plane, according to \cite[Theorem 4.3.7 and Proposition
4.4.6]{obs_book}. Thus, it follows that $\GGG_P$ is bounded on some
right half-plane. Now it follows from Proposition \ref{Macron} that
$(A+P,B,C,D)$ is well-posed.
\end{proof}

We mention that the above proposition remains valid for a time
varying $P:[0,\infty)\rarrow\Lscr(X)$, as long as it is strongly
continuous. This is much harder to prove, see
\cite[Theorems 4.2 and 5.3]{ChWe:15}.

We introduce a special class of well-posed systems, following the
terminology in \cite{TuWe_14_survey}, \cite{art12}, \cite{WeCu_97},
\cite{ZhWe:11} and many other papers. We do this because our systems
\dref{Sch} and \dref{Sch_unp} fall into this category (as we shall
see), and we will use tools developed for such systems.

\begin{definition} \label{regular} {\rm
Let $(A,B,C,D)$ be a well-posed system node on $(U,X,Y)$, with
transfer function $\GGG$ (see \dref{Mattis_resigns}). We say that this
system is {\em regular} if the limit $D_0 v=\lim_{\l\rarrow\infty,\m\l
\in\rline}\GGG(\l)v$ exists, for each $v\in U$. In this case,
$D_0\in\Lscr(U,Y)$ is called the {\em feedthrough operator} of the
system.}
\end{definition}

\begin{proposition} \label{Savta}
Suppose that the compatible system node $(A,B,C,D)$ on $(U,X,Y)$
is well-posed, and let $\GGG$ be its transfer function. Recall $\CL$
from {\rm\dref{MillionDollars}} and the space $Z$ introduced in {\rm
\dref{yellow_vests}}. We have \m $Z\subset D(\CL)$ if and only if the
system is regular.

If the system is regular, then the quadruple $(A,B,C,D)$ may be
replaced with the equivalent quadruple $(A,B,\CL,D_0)$ (where $D_0$ is
the feedthrough operator of the system), in the sense that this new
quadruple has the same system operator $S$ and the same transfer
function.
\end{proposition}

The following proposition recalls some properties of output feedback
for regular linear systems (for the proof see \cite{art12}). In the
proposition we make the simplifying assumption $\Kf D=0$ (true in our
application in Sect. 4) that greatly simplifies the formulas.

\begin{proposition} \label{Ruth}
Let $(A,B,C,D)$ be a regular linear system on $(U,X,Y)$, with transfer
function $\GGG$. Assume that the feedthrough operator of this system
is $D$, and let $\Kf\in\Lscr(Y,U)$. We assume that the function
$I-\Kf\GGG(s)$ has a uniformly bounded inverse for all $s$ in some
right half-plane, and $\Kf D=0$. Then $(A_{cl},B,(I+D\Kf)\CL,D)$ is a
regular linear system on $(U,X,Y)$, called the {\em closed-loop
system} corresponding to $(A,B,C,D)$ with the output feedback operator
$\Kf$. Here \vspace{-1mm}
$$ A_{cl} \m=\m A+B\Kf\CL \m,\qquad D(A_{cl}) \m=\m \left\{x\in Z\ |\
   Ax+B\Kf\CL x\in X \m\right\} \m.$$
(The sum \m $Ax+B\Kf\CL x$ is computed in $X_{-1}$.) In particular,
$(I+D\Kf)\CL$ is an admissible observation operator for the
semigroup generated by $A_{cl}$.
\end{proposition}

Intuitively, the closed-loop system $(A_{cl},B,(I+D\Kf)\CL,D)$ is
obtained from the original system $(A,B,C,D)$ via the output feedback
$u=\Kf\m y+ \rho$ (where $\rho$ is the new input function). The
transfer function of the closed-loop system is \m $\GGG_{cl}=\GGG
(I-\Kf\GGG)^{-1}=(I-\GGG\Kf)^{-1}\GGG$.

Let $\GGG$ be a function defined on some domain in $\cline$ that
contains a right half-plane, with values in a normed space.
Following \cite{ZhWe:11}, we say that $\GGG$ is {\em strictly proper}
if \vspace{-1mm}
$$ \lim_{\Re s\rarrow\infty} \|\GGG(s)\| \m=\m 0 \m, \  \mbox{
   uniformly with respect to }\ \Im s \m.$$
In other words, there exists an $\alpha\in\rline$ and a continuous
function $\beta:(\alpha,\infty)\rarrow(0,\infty)$ such that
\vspace{-1mm}
\begin{equation} \label{Raimar}
   \|\GGG(s)\| \m\leq\m \beta(\Re s) \FORALL s\in\cline_\alpha \
   \mbox{ and }\ \lim_{\xi\rarrow\infty}\beta(\xi) \m=\m 0 \m.
\end{equation}
The notation $\cline_\alpha$ has been introduced in Proposition
\ref{Macron}. The above concept generalizes the well-known one of
strictly proper rational transfer function. A well-posed system node
is called {\em strictly proper} if its transfer function is strictly
proper. Clearly such systems are regular and their feedthrough
operator is zero.

The following proposition shows a curious property of certain
semigroup generators $A$: if $B$, $C$ and $D$ are such that
$(A,B,C,D)$ is a compatible system node, then the admissibility of $B$
and $C$ for the semigroup generated by $A$ implies the well-posedness
of $(A,B,C,D)$. Moreover, it turns out that the compatible system
node $(A,B,\CL,0)$ is strictly proper.

\begin{proposition} \label{Gail}
Let $X=l^2$, $a>0$ and let the operator $A:D(A)\rarrow X$ be defined
on sequences $x=(x_k)$ ($k\in\nline$) by \vspace{-2mm}
$$ \m\ \ \ \ \ (Ax)_k \m=\m ia k^2 x_k \m,\qquad D(A) \m=\m \left\{ x
   \in l^2\ \left|\ \sum_{k\in\nline} k^4 |x_k|^2 < \infty \right.
   \right\} \m.\vspace{-1mm}$$
Then $A$ is the generator of the diagonal unitary operator group
\vspace{-1mm}
$$(\tline_t x)_k \m=\m e^{ia k^2 t} x_k \FORALL x\in X,\ t\geq 0 \m.$$

Let $B\in X_{-1}$ be an admissible control operator for $\tline$ (for
the input space $\cline$) and let the bounded linear functional $C:X_1
\rarrow\cline$ be an admissible observation operator for $\tline$ (for
the output space $\cline$).

Then $(A,B,\CL,0)$ is a compatible system node that is well-posed and
strictly proper.
\end{proposition}

\begin{proof}
The fact that $A$ generates the indicated operator group $\tline$ is
easy and standard material in semigroup theory, see e.g. \cite
[Proposition 2.6.5]{obs_book}. Let $\left\{e_1,e_2,e_3,\ldots
\right\}$ be the standard orthonormal basis of $l^2$. We denote by
$b_k$ and $c_k$ the components of $B$ and $C$, respectively:
\vspace{-1mm}
$$ b_k \m=\m \langle B,e_k\rangle, \qquad c_k \m=\m C e_k \FORALL
   k\in\nline \m.\vspace{-1mm}$$
It follows from the Carleson measure criterion for admissibility
(see e.g. \cite[Proposition 5.3.5]{obs_book}) that the sequences
$(b_k)$ and $(c_k)$ are bounded. We want to check that for some
(hence for every) $s\in\cline_0$ we have $(sI-A)^{-1}B\in D(\CL)$.
For this, we compute \vspace{-1mm}
$$ \begin{array}{l}\disp\lim_{\l\rarrow\infty} C\l(\l I-A)^{-1}
   (sI-A)^{-1}B \crr\disp \m=\m \lim_{\l\rarrow\infty} 
   \sum_{k\in\nline} \frac{b_k c_k \l}{(\l-iak^2)(s-iak^2)} \m=\m 
   \sum_{k\in\nline} \frac{b_k c_k}{s-iak^2} \m.\end{array}$$
This shows that indeed $(sI-A)^{-1}B\subset D(\CL)$, which implies
that $Z\subset D(\CL)$, and \vspace{-2mm}
$$ \CL(sI-A)^{-1}B \m=\m \sum_{k\in\nline} \frac{b_k c_k}{s-iak^2}
   \m.\vspace{-1mm}$$
Hence, for any $s\in\cline_0$ we have, denoting $\theta=\Re s/a$ and
$\o=\Im s$, \vspace{-1mm}
$$ |\CL(sI-A)^{-1}B| \m\leq\m \sum_{k\in\nline} |b_k c_k| \frac{1}
   {\left|\theta a+i(\o-ak^2)\right|} \m.\vspace{-2mm}$$
Using the elementary inequality $|\tilde a+i\tilde b|\geq(|\tilde a|+
|\tilde b|)/\sqrt{2}$ \ (for any $\tilde a,\tilde b\in\rline$),
we get \vspace{-1mm}
\begin{equation} \label{Helen}
  |\CL(sI-A)^{-1}B| \m\leq\m m\sqrt{2} \sum_{k\in\nline} \frac{1}
  {\theta a+|\o-ak^2|} \m=\m \frac{m\sqrt{2}}{a} \sum_{k\in\nline}
  \frac{1}{\theta+|\mu-k^2|} \m,\vspace{-1mm}
\end{equation}
where $m=\sup|b_k c_k|$ and $\mu=\o/a$. Considering the case $\mu\leq
0$, we get \vspace{-1mm}
\begin{equation} \label{Damascus_airport}
  |\CL(sI-A)^{-1}B| \m\leq\m \frac{m\sqrt{2}}{a} \sum_{k\in\nline}
  \frac{1}{\theta+k^2} \ \mbox{ for }\ \Re s=\theta a \m,\ \
  \Im s < 0 \m.\vspace{-2mm}
\end{equation}
Now consider the case $\mu>0$, and denote by $k_\mu$ the largest
integer $k$ satisfying $k^2\leq\mu$. We decompose \vspace{-2mm}
\begin{equation} \label{Gatwick}
  \sum_{k\in\nline} \frac{1}{\theta+|\mu-k^2|} \m=\m
  \sum_{1\leq k\leq k_\mu} \frac{1}{\theta+\mu-k^2} +
  \sum_{k> k_\mu} \frac{1}{\theta+k^2-\mu} \m.
\end{equation}
It is very easy to see that the second sum on the right side above is
bounded by \m $\sum_{k\in\nline} 1/(\theta+k^2)$. For the first sum
we do the change of discrete variables $j=k_\mu-k$, obtaining
$$ \begin{array}{l}\disp\sum_{1\leq k\leq k_\mu} \frac{1}{\theta+
   \mu-k^2} \m=\m \sum_{1\leq j\leq k_\mu} \frac{1}{\theta+\mu
   - k_\mu^2 + 2k_\mu j-j^2}\crr\disp
   \m\leq\m \sum_{1\leq j\leq k_\mu} \frac{1}{\theta+2k_\mu j-j^2}
   \m\leq\m \sum_{1\leq j\leq k_\mu} \frac{1}{\theta+j^2} \m.
   \end{array}$$
Combining this with our earlier estimate for the second sum in
\dref{Gatwick}, it follows that \vspace{-2mm}
$$ \sum_{k\in\nline} \frac{1}{\theta+|\mu-k^2|} \m\leq\m 2
   \sum_{k\in\nline} \frac{1}{\theta+k^2} \ \mbox{ for }\ \mu>0 \m.$$
This, together with \dref{Helen} and \dref{Damascus_airport} implies
that, for any $\theta>0$, \vspace{-1mm}
\begin{equation} \label{Dimona}
  |\CL(sI-A)^{-1}B| \m\leq\m \frac{2\sqrt{2}m}{a} \sum_{k\in\nline}
  \frac{1}{\theta+k^2} \ \mbox{ for }\ \Re s=\theta a \m.
\end{equation}
If we denote the righ-hand side of \dref{Dimona} with $\beta(\Re s)$
and compare with \dref{Raimar}, we see that $\CL(sI-A)^{-1}B$ is
strictly proper. In particular, this transfer function is bounded on
any half-plane $\cline_\alpha$ with $\alpha>0$. According to the last
part of Proposition \ref{Macron}, $(A,B,\CL,0)$ is well-posed.
\end{proof}

\begin{corollary} \label{Merkel}
Let $X$ be a Hilbert space, let $A:D(A)\rarrow X$ be the generator of
an operator semigroup $\tline$ on $X$, let $B\in\Lscr(\cline^m,
X_{-1})$ be an admissible control operator for $\tline$ and let $C\in
\Lscr(X_1,\cline^p)$ be an admissible observation operator for
$\tline$. Assume that $A$ is diagonalizable, meaning that there is a
Riesz basis $(\phi_k)$ in $X$ ($k\in\nline$) consisting of
eigenvectors of $A$, and the corresponding eigenvalues $\mu_k$ satisfy
\vspace{-1mm}
$$ \mu_k \m=\m ia k^2 + \Oscr(1),\ \mbox{ where }\ a>0 \m.$$

Then $(A,B,\CL,0)$ is a compatible system node that is well-posed
and strictly proper.
\end{corollary}

Indeed, this follows from Propositions \ref{Trump} and \ref{Gail}.

\section{Properties of the system to be controlled} 

\ \ \ We want to reformulate the equations \dref{Sch} and
\dref{Sch_unp} in the abstract operator theory framework. For this,
first we introduce a semigroup generator on $\hline$, a bounded
perturbation of $A$ from \dref{A-def}: \vspace{-1mm}
\begin{equation} \label{Ah-def}
   A_h f \m=\m Af + h f \FORALL f\in D(A_h) \m=\m D(A).
\end{equation}
We define the operators $B_l,B_r$ as follows: \vspace{-1mm}
$$B_l \m=\m i\delta(\cdot),\qquad B_r \m=\m -i\delta(\cdot-1).$$
Here $\delta$ is the Dirac mass. We denote the adjoints of $A$,
$B_l$ and $B_r$ by $A^*$, $B_l^*$ and $B_r^*$, respectively, and it
is easy to check that \vspace{-2mm}
$$ \begin{array}{l}\disp A^*f \m=\m if'',\quad D(A^*) \m=\m \{f\in
   H^2(0,1)\ |\ f'(0)=iqf(0),\ f'(1)=0\},\crr B_l^*f \m=\m -if(0),\
   \ B_r^*f \m=\m if(1) \FORALL f\in D(A^*). \end{array}$$
The operators $B_l$ and $B_r$ are the control operators that
correspond to the inputs $d_2$ and $u$ in the boundary control
systems \dref{Sch} as well as \dref{Sch_unp}. This can be checked
using \cite[Remark 10.1.6]{obs_book}.

Define $C_m\in\Lscr(H^1(0,1),\cline)$ by $C_m f=f(1)$. Then
\dref{Sch} can be rewritten in the abstract form \vspace{-1mm}
\begin{equation} \label{Sch-ab}
   \left\{ \begin{array}{l} \dot{z}(\cdot,t) \m=\m A_hz(\cdot,t) +
   g(\cdot)d_1(t) + B_l d_2(t) + B_r u(t) \m,\\ \m\ \ y(t) \m=\m C_e
   [z(\cdot,t)] \m,\qquad y_m(t) \m=\m C_m[z(\cdot,t)] \m,
   \end{array} \right.
\end{equation}
which corresponds to the compatible system node$(A_h,[g(\cdot)\
B_l\ B_r],\sbm{C_e\\ C_m},0)$ on $(\cline^3,\hline,\cline^2)$. It is
easy to check that for this system node, the space $Z$ from
\dref{yellow_vests} is given by \vspace{-2mm}
\begin{equation} \label{Samsung_to_Hanna}
  Z \m=\m H^2(0,1) \m.
\end{equation}
The equivalence between \dref{Sch} and \dref{Sch-ab} means that they
have the same classical solutions, and this equivalence can be
checked using the techniques in \cite[Sect.~10.1]{obs_book}.

Similarly, the system \dref{Sch_unp} can be rewritten in the
abstract form \vspace{-1mm}
\begin{equation} \label{Theresa_May}
   \left\{ \begin{array}{l} \dot{z}(\cdot,t) \m=\m A z(\cdot,t)
   + B_l d_2(t) + B_r u(t) \m,\\ \m\ \ y(t) \m=\m C_e[z
   (\cdot,t)] \m,\qquad y_m(t) \m=\m C_m [z(\cdot,t)] \m,
   \end{array} \right.
\end{equation}
which corresponds to the compatible system node $(A,[B_l\ B_r],
\sbm{C_e\\ C_m},0)$ on $(\cline^2,\hline,\cline^2)$. For this system
node, the space $Z$ is again given by \dref{Samsung_to_Hanna}.

\begin{lemma} \label{lem-A}
Let $A$ be defined by {\rm\dref{A-def}}. Then $A^{-1}$ exists and it
is compact. Hence, $\sigma(A)$, the spectrum of $A$, consists of
isolated eigenvalues of finite algebraic multiplicity. All eigenvalues
of $A$ are located in a vertical strip, they have positive real parts
and there exists a sequence of eigenfunctions of $A$, which forms a
Riesz basis for $\hline$. Therefore, $A$ generates an operator group
$\tline$ on $\hline$.

The observation operator $C_m$ is admissible for the group $\tline$.
\end{lemma}

\begin{proof}
A straightforward computation shows that $A$ has a bounded inverse on
$\hline$ and \vspace{-3mm}
$$ \m\ \ \ (A^{-1}\phi)(x) \m=\m \frac{-(i+qx)\int_0^1 \phi(y)\dd y}
   {iq} - i\int_0^x (x-y)\phi(y)\dd y.$$
Since the embedding of $H^1(0,1)$ into $L^2[0,1]$ is compact, it
follows that $A^{-1}$ is compact. This implies that $\sigma(A)$
consists of isolated eigenvalues of finite algebraic multiplicity. It
is easy to verify that \m $\Re\langle Af,f\rangle=q |f(0)|^2\geq 0$,
which implies that all the eigenvalues of $A$ have non-negative real
parts. Next, we show that there is no eigenvalue on the imaginary
axis. Otherwise, suppose that $Af=i\beta f$ with $\beta\in\rline$ has
a nonzero solution, i.e., \vspace{-1mm}
\begin{equation} \label{f-equ-eig}
   f''(x) \m=\m -\beta f(x),\ \ f'(0) \m=\m -iqf(0),\ \ f'(1) \m=\m 0.
\end{equation}
Multiplying the first equation of \dref{f-equ-eig} with
$\overline{f(x)}$ (the conjugate of $f(x)$) and integrating over
$[0,1]$, it follows from the boundary condition that \vspace{-2mm}
$$iq|f(0)|^2-\int_0^1|f'(x)|^2\dd x=-\beta\int_0^1|f(x)|^2\dd x,$$
which, jointly with $q>0$ and taking imaginary part, gives $f(0)=0$.
By the second equation of \dref{f-equ-eig}, $f'(0)=0$. Thus,
\dref{f-equ-eig} has only the zero solution, a contradiction.
Therefore, all the eigenvalues of $A$ have positive real parts.

Now we consider the eigenvalue problem $Af=\mu f$ and let
$\mu=-i\l^2$, that is \vspace{-2mm}
$$ \phi''(x) \m=\m \l^2\phi(x),\ \ \phi'(0) \m=\m -iq\phi(0),\ \
   \phi'(1) \m=\m 0,$$
to yield \vspace{-4mm}
\begin{equation} \label{phi-eigf}
   \phi(x) \m=\m \frac{\l-iq}{\l+iq}e^{\lambda x}+e^{-\l x},
\end{equation}
where $\l\in\cline$ satisfies \vspace{-1mm}
\begin{equation} \label{eig-lam1}
   e^{2\l} \m=\m \frac{\l+iq}{\l-iq} \m=\m 1+\frac{2iq}{\l-iq}
   \m=\m 1+\frac{2iq}{\l}+\Oscr(|\l|^{-2})\mbox{   as  }|\l|
   \to\infty.
\end{equation}
Thus, we have \vspace{-2mm}
$$\l_n \m=\m n\pi i+\Oscr(n^{-1}), \qquad n\in\nline.$$
Substituting this into \dref{eig-lam1}, we get that for this
specific case, $\Oscr(n^{-1})=q/(n\pi)+\Oscr(n^{-2})$, hence
\vspace{-2mm}
$$ \l_n \m=\m n\pi i+\frac{q}{n\pi}+\mathcal{O}(n^{-2}),
   \qquad n\in\nline \m.$$
It follows from here and \dref{phi-eigf} that the asymptotic
expressions for eigenpairs of $A$ are \vspace{-2mm}
\begin{equation} \label{Haley}
   \left\{\begin{array}{l} \mu_n=2q+i(n\pi)^2+\Oscr(n^{-2}),
   \\ \phi_n(x) \m=\m \cos(n\pi x)+\Oscr(n^{-1}),
   \end{array}\right.
\end{equation}
which implies that all the eigenvalues $\mu_n$ of $A$ are located in a
vertical strip and the corresponding eigenvectors $\phi_n$ are
quadratically close to an orthonormal basis. By a theorem known as
``Bari's theorem'', see \cite[Theorem 6.3]{Guo_2001} or \cite[Theorem
2.4]{XuWe_2011}, $\{\phi_n\}$ forms a Riesz basis for $\hline$. This
shows that the spectrum-determined growth condition holds for
$A$. Thus, $A$ generates an operator semigroup $\tline$ and
$\|\tline_t\|\leq Le^{\o t}$ with $\o=\sup\{\Re\l\m| \ \l\in
\sigma(A)\}$ for some $L\geq 1$. Similarly, we can show that $-A$
generates an operator semigroup. By \cite[Proposition 2.7.8]
{obs_book}, $\tline$ can be extended to a group.

We show that $C_m$ is admissible for the group $\tline$. Denote
$c_n=C_m\phi_n$, $n\in\nline$, then according to \dref{Haley} we have
$c_n=\cos(n\pi)+\mathcal{O}(n^{-1})$. The eigenvalues $\mu_n$ are in
a vertical strip and for large enough $n$, the distance between their
imaginary parts is bounded from below by a positive number. Thus, the
admissibility of $C_m$ follows from the simple version of the Carleson
measure criterion applicable for diagonal operator groups, see
\cite[Proposition 5.3.5]{obs_book}.
\end{proof}

\begin{remark} {\rm
By Lemma \ref{lem-A}, all the eigenvalues of $A$ have positive real
parts. So, if the values $h(x)$ in \dref{Sch} are non-negative or if
its sup norm is sufficiently small, then also the eigenvalues of $A_h$
have positive real parts. This is why we call the system \dref{Sch}
anti-stable. This situation is different from the unstable case in
\cite{Deu2015auto}, where there are at most finitely many unstable
eigenvalues for the system to be controlled. The system \dref{Sch} is
different also from the one in \cite{Liu_2018}, where the
Schr\"{o}dinger equation is essentially exponentially stable when
the disturbance vanishes.}
\end{remark}

\begin{proposition} \label{Brighitte}
The operators $B_l,B_r$ are admissible control operators for
$\tline$. Therefore, for any initial state $z(0)=z_0\in\hline$ and
any $d_2,u\in L^2_{loc}[0,\infty)$, the first equation in {\rm
\dref{Theresa_May}} admits a unique solution in $\hline_{-1}$ (in
the sense of {\rm\cite[{\it Definition} 4.1.1]{obs_book}}) and $z\in
C([0,\infty);\hline)$.

Moreover, if $d_2,u\in H^1_{loc}(0,\infty)$ are such that $Az_0+B_l
d_2(0)+B_r u(0)\in\hline$, then the solution $z$ satisfies
\vspace{-2mm}
\begin{equation} \label{Enrique}
  z \m\in\m C([0,\infty);Z) \cap C^1([0,\infty);\hline) \m.
\end{equation}
In this case, the functions $y$ and $y_m$ can be defined by the
second equation in {\rm\dref{Theresa_May}} and $(z,\sbm{d_2\\ u},
\sbm{y\\ y_m})$ is a classical solution of {\rm\dref{Theresa_May}}
and also of {\rm\dref{Sch_unp}}.
\end{proposition}

Recall that $Z$ appearing above is given by \dref{Samsung_to_Hanna}.
We remark that the condition $A_h z_0+B_l d_2(0)+B_r u(0)\in\hline$
appearing above is equivalent to \vspace{-1mm}
$$ z_0 \m\in\m H^2(0,1) \m,\qquad \frac{\dd}{\dd x} z_0(0) \m=\m -iq
   z(0) + d_2(0) \m,\qquad \frac{\dd}{\dd x} z_0(1) \m=\m u(0) \m.$$
This can be verified using the techniques of boundary control systems
in \cite[Sect.~10.1]{obs_book}.

\begin{proof} We prove the admissibility of $B_l$ for $\tline$. For
this, recall from \cite[Theorem 4.4.3]{obs_book} that it suffices to
show that $B^*_l$ an admissible observation operator for the adjoint
semigroup $\tline^*$. This is equivalent to showing that (i) $B_l^*
{A^*}^{-1}$ is a bounded operator on $\hline$ and (ii) for each $T>0$
there exists $M_T >0$ such that for every initial state, the output
signal $\eta$ of the system (defined for $t\geq 0$) \vspace{-2mm}
\begin{equation} \label{z-equ}
   \left\{ \begin{array}{l}\disp z_t(x,t) \m=\m iz_{xx}(x,t),\ \ \
   x\in(0,1),\\ z_x(0,t) \m=\m iqz(0,t),\ \ \ z_x(1,t) \m=\m 0,\\
   \eta(t) \m=\m z(0,t),\end{array}\right.
\end{equation}
satisfies
$$ \int_0^T |\eta(t)|^2\dd t \m\leq\m M_T E(0),\quad \mbox{where}
   \qquad E(t) \m=\m \half \|z(\cdot,t)\|_\hline^2 \m.$$
A simple computation shows that $A^*$ has bounded inverse
on $\hline$ and \vspace{-2mm}
$$ \left\{\begin{array}{l}\disp {A^*}^{-1}\phi \m=\m \frac{(-i+qx)
   \int_0^1 \phi(y)\dd y}{-iq}-i\int_0^x(x-y)\phi(y)\dd y,\\ \disp
   B_l^*{A^*}^{-1}\phi \m=\m -\frac{i}{q}\int_0^1\phi(y)\dd y.
   \end{array}\right.$$
Hence $B_l^*{A^*}^{-1}$ is bounded on $\hline$. We differentiate $E$
with respect to $t$ along the solution of \dref{z-equ} to obtain \m
$\dot{E}(t)=q|\eta(t)|^2$, which, together with Lemma \ref{lem-A},
gives \vspace{-2mm}
$$ \int_0^T |\eta(t)|^2\dd t \m=\m \frac{1}{q}[E(t)-E(0)] \m\leq\m
   \frac{1}{q}[1+Le^{\o T}]E(0),$$
where $\o,L$ are as in the proof of Lemma \ref{lem-A}. Thus, $B_l$ is
admissible.

The proof of the fact that $B_r$ is an admissible control operator for
$\tline$ is similar. The statement about unique and continuous
solutions of \dref{Theresa_May} follows from \cite[Proposition
4.2.5]{obs_book}. Finally, the statement for $d_2,u\in H^1_{loc}(0,
\infty)$ follows from \cite[Proposition 4.2.10]{obs_book}.
\end{proof}

There is a similar statement for the original system \dref{Sch},
formulated abstractly in \dref{Sch-ab}:

\begin{corollary} \label{ProAB}
The operator $A_h$ from {\rm\dref{Ah-def}} generates an operator group
$(e^{A_ht})_{t\in\rline}$ on $\hline$ and $B_l,B_r$ are admissible
control operators for this operator group. Therefore, for any initial
state $z(0)=z_0\in\hline$ and any $d_1,d_2,u\in L^2_{loc}[0,\infty)$,
the first equation in {\rm\dref{Sch-ab}} admits a unique solution in
$\hline_{-1}$ (in the sense of {\rm\cite[{\it Definition} 4.1.1]
{obs_book}}) and $z\in C([0,\infty);\hline)$.

Moreover, if $d_1,d_2,u\in H^1_{loc}(0,\infty)$ are such that $A_h
z_0+B_l d_2(0)+B_r u(0)\in\hline$, then the solution $z$ satisfies
{\rm\dref{Enrique}}. In this case, the functions $y$ and $y_m$ can be
defined by the second equation in {\rm\dref{Sch-ab}} and
$(z,\sbm{d_1\\ d_2\\ u}, \sbm{y\\ y_m})$ is a classical solution of
{\rm\dref{Sch-ab}}, and also of {\rm\dref{Sch}}.
\end{corollary}

\begin{proof} By Lemma \ref{lem-A} and the boundedness of $h$, it is
clear that $A_h$ generates a strongly continuous operator group on
$\hline$ (this follows, for instance, by applying \cite[Theorem
2.11.2]{obs_book} to $A_h$ and also to $-A_h$). Since $A_h$ is a
bounded perturbation of $A$, according to \cite[Corollary 5.5.1]
{obs_book}, $B_l$ and $B_r$ are admissible control operators also 
for $(e^{A_h t})_{t\geq 0}$. The end of the proof is now the same as
for Proposition \ref{Brighitte}.
\end{proof}

\begin{proposition} \label{keso_van}
The compatible system node $(A_h,[g(\cdot)\ B_l\ B_r],\sbm{C_e\\
C_m}\nm,0)$ (which corresponds to the equations {\rm\dref{Sch-ab}})
is well-posed. Similarly, the compatible system node $(A,[B_l\ B_r],
\sbm{C_e\\ C_m}\nm,0)$ (which corresponds to the equations
{\rm\dref{Theresa_May}}) is well-posed. If we replace $C_e$ with
$C_{e\L}$ (defined as in {\rm\dref{MillionDollars}}), then both of
these system nodes become strictly proper (hence, all these systems
are regular).
\end{proposition}

\begin{proof} We start with the compatible system node $(A,[B_l\ B_r],
\sbm{C_e\\ C_m}\nm,0)$, whose control operator $B=[B_l\ B_r]$ is known
to be admissible from Proposition \ref{Brighitte} and whose
observation operator $C=\sbm{C_e\\ C_m}$ is known to be admissible
from our assumption on $C_e$ in Sect.~1 and from Lemma \ref{lem-A}. We
know from \dref{Haley} that $A$ satisfies the assumptions of Corollary
\ref{Merkel}. Hence, according to this corollary, $(A,[B_l\ B_r],
\sbm{C_e\\ C_m}\nm,0)$ is well-posed. According to Proposition
\ref{Trump}, $(A_h,[B_l\ B_r],\sbm{C_e\\ C_m}\nm,0)$ is also
well-posed. The well-posedness of this system node will not be
affected if we add another bounded component to its control operator,
changing it to $[g(\cdot)\ B_l\ B_r]$.

For the operator $C_m$ it is not difficult to show that its
extension $C_{m\L}$, when restricted to $Z$, is again $C_m$. However
for $C_e$, which has not been specified, we do not know if this is the
case. However, after having replaced $C_e$ with $C_{e\L}$, we can
apply Corollary \ref{Merkel} to $(A,[B_l\ B_r],\sbm{C_{e\L}\\ C_m}\nm,
0)$ to conclude that its transfer function $\GGG$ is strictly proper.
For the transfer function $\GGG_P$ of $(A_h,[B_l\ B_r],\sbm{C_{e\L}\\
C_m}\nm,0)$ we use the identity \dref{Gad_Hareli}, with $P$ being the
operator of pointwise multiplication with the function $h$, so that
$A_h=A+P$. Since the functions $C(sI-A)^{-1}$ and $(sI-A-P)^{-1}B$
(with $C=\sbm{C_e\\ C_m}$ and $B=[B_l\ B_r]$) are known to be
strictly proper, see for instance \cite[Theorem 4.3.7 and
Proposition 4.4.6]{obs_book}, it follows that $\GGG_P$ is strictly
proper. Finally, when adding the extra component to $B$, replacing the
earlier $B$ with $[g(\cdot)\ B_l\ B_r]$, then the transfer function
remains strictly proper, because the new component $g(\cdot)$ is a
bounded control operator.
\end{proof}

\section{State feedback regulation} 

\ \ \ In this section we will construct a state feedback operator
that solves the regulator problem. We denote $\Om=\{(x,\xi)\in\rline^2
\m|\ 0\leq \xi\leq x\leq 1\}$. First we introduce the {\em
backstepping transformation}
\begin{equation} \label{trans-z-u}
  v(x,t) \m=\m \Fscr[z(\cdot,t)](x,t):=z(x,t)-\int_0^x k(x,\xi)z(\xi,
  t)\dd \xi,
\end{equation}
where the kernel function $k:\Om\to\rline$ satisfies, for some fixed
$c_s>0$,
\begin{equation} \label{k-pde}
   \left\{\begin{array}{l} k_{xx}(x,\xi)-k_{\xi\xi}(x,\xi) \m=\m
   (h(\xi)+c_s)i k(x,\xi),\crr \disp k_\xi(x,0) + qik(x,0)=0,\crr
   \disp k(x,x) \m=\m -\frac{i}{2}\int_0^x(h(\xi)+c_s)\dd\xi-qi.
   \end{array}\right.
\end{equation}
By \cite[Theorem 2.1]{AS-Krstic-book}, the above system of equations
has a unique solution $k\in C^2(\Om)$. It can be shown \cite[Theorem
2.2]{AS-Krstic-book} that this transformation is boundedly
invertible, and \vspace{-2mm}
$$ \Fscr^{-1}[v(\cdot,t)](x,t) \m=\m v(x,t) + \int_0^x K(x,\xi)v(\xi,
   t) \dd\xi \m,\vspace{-1mm}$$
where the kernel function $K$ is also in $C^2(\Om)$. It is easy to
see from \dref{trans-z-u} and the above formula that $\Fscr$ and
$\Fscr^{-1}$ leave $C^1$ and $H^2$ functions invariant: 
\vspace{-1mm}
$$ \Fscr C^1[0,1] \m\subset\m C^1[0,1], \qquad
   \Fscr^{-1} C^1[0,1] \m\subset\m C^1[0,1],$$
\vspace{-4mm}
$$ \Fscr H^2(0,1) \m\subset\m H^2(0,1), \qquad
   \Fscr^{-1} H^2(0,1) \m\subset\m H^2(0,1) \m.$$

The proposed {\em state feedback law} (applied to classical solutions
of \dref{Sch}) is given by a continuous linear functional $F$ defined
on $H^2(0,1)$ plus a term applied to the exosystem state $w$:
\vspace{-1mm}
\begin{equation} \label{state-con}
   u(t) \m=\m F[z(\cdot,t)] + m_w^\top w(t) \m=\m k(1,1)z(1,t) +
   \int_0^1 k_x(1,\xi) z(\xi,t)\dd\xi + m_w^\top w(t),
\end{equation}
where $m_w^\top$ is a constant vector to be determined later. With
this feedback, the first equation in \dref{Sch-ab} becomes
\vspace{-1mm}
\begin{equation} \label{A+BF}
   \dot{z}(\cdot,t) \m=\m (A_h+B_r F)z(\cdot,t) + g(\cdot)d_1(t) +
   B_l d_2(t) + B_r m_w^\top w(t) \m.
\end{equation}
Under the state feedback \dref{state-con}, the classical solutions
of \dref{A+BF} must satisfy the following equations (which are
obtained by substituting \dref{state-con} into \dref{Sch}):
\vspace{-1mm}
\begin{equation} \label{Sch-state-close-z}
   \left\{\begin{array}{l} z_t(x,t) \m=\m -iz_{xx}(x,t)+h(x)z(x,t)+
   g(x)d_1(t), \crr\disp z_x(0,t) \m=\m -iqz(0,t)+d_2(t), \crr\disp
   z_x(1,t) \m=\m k(1,1)z(1,t)+\int_0^1k_x(1,\xi)z(\xi,t)\dd\xi +
   m_w^\top w(t),\crr\disp z(x,0) \m=\m z_0(x),\qquad y(t) \m=\m
   C_e \left[z(\cdot,t)\right] \m,\qquad y_m(t) \m=\m z(1,t) \m.
   \end{array}\right.
\end{equation}
Using the transformation \dref{trans-z-u} and omitting $y_m$, the
system \dref{Sch-state-close-z} becomes
\begin{equation} \label{Sch-state-close-z-u}
   \left\{\begin{array}{l} v_t(x,t) \m=\m -iv_{xx}(x,t)-c_sv(x,t)+
   \Fscr[g](x)d_1(t)-k(x,0)d_2(t),\crr\disp v_x(0,t)=d_2(t), \qquad
   v_x(1,t)=m_w^{\top}w(t), \crr\disp v(x,0) \m=\m z_0(x)-\int_0^x
   k(x,\xi)z_0(\xi)\dd\xi \m,\qquad y(t) \m=\m C_e \Fscr^{-1} \left[
   v(\cdot,t)\right] \m. \end{array}\right.
\end{equation}
In order to find the constant vector $m_w$ in \dref{state-con}, we
introduce the {\em error transformation} \vspace{-1mm}
\begin{equation} \label{tu-u-trans}
   \widetilde{v}(x,t) \m=\m v(x,t)-m(x)^{\top}w(t).
\end{equation}
We are searching for a function $m\in C^2([0,1];\rline^{n_w})$ for the
transformation \dref{tu-u-trans} so that the first three equations in
\dref{Sch-state-close-z-u} can be converted into the following (with
$x\in(0,1)$ and $t\geq 0$):
\begin{equation} \label{ORSch-tu}
   \left\{\begin{array}{l}\disp \widetilde{v}_{t}(x,t) \m=\m -i
   \widetilde{v}_{xx}(x,t)-c_s\widetilde{v}(x,t),  \crr\disp
   \widetilde{v}_x(0,t)=0, \qquad \widetilde{v}_x(1,t) \m=\m 0 \m.
   \end{array}\right.
\end{equation}
In other words, $\dot{\widetilde{v}}=\left(\aline-c_s I\right)
\widetilde{v}$, where $\aline$ is the following skew-adjoint operator:
\vspace{-2mm}
\begin{equation} \label{Alina}
  \aline f \m=\m -if'' \ \ \mbox{ with }\ \ D(\aline) \m=\m
  \{f\in H^2(0,1)\m|\ f'(0)=f'(1)=0\} \m.
\end{equation}
This $\aline$ is a simplified version of $A$ from \dref{A-def} that
corresponds to $q=0$. Thus, the differential equation of $\widetilde
{v}$ is exponentially stable in $\hline$.

Substituting \dref{tu-u-trans} into the first part of \dref{ORSch-tu},
we get
\begin{equation} \label{sol-m1}
   \begin{array}{l}\disp 0 \m=\m \widetilde{v}_{t}(x,t)+i
   \widetilde{v}_{xx}(x,t)+c_s\widetilde{v}(x,t)\crr\disp
   =\m v_{t}(x,t)-m(x)^\top Sw(t)+iv_{xx}(x,t)-im''(x)^\top w(t)
   +c_s v(x,t)-c_s m(x)^\top w(t)\crr\disp =\m -\left[ im''(x)^\top+
   m(x)^\top S+c_s m(x)^\top-\Fscr[g](x)p_1^{\top}+k(x,0)p_2^\top
   \right]w(t). \end{array}
\end{equation}
Here we have used $p_1,p_2$ from \dref{dis-produce}. Substituting
\dref{tu-u-trans} into the second part of \dref{ORSch-tu}, we get
(using \dref{Sch-state-close-z-u}) \vspace{-3mm}
\begin{equation} \label{sol-m2}
   \begin{array}{l}\disp 0 \m=\m \widetilde{v}_x(0,t)
   \m=\m v_{x}(0,t)-m'(0)^{\top}w(t) \m=\m \left[ p_2^{\top}-m'(0)
   ^{\top}\right]w(t). \end{array}
\end{equation}
Substituting \dref{tu-u-trans} into the third part of \dref{ORSch-tu},
we have (using \dref{Sch-state-close-z-u})
\begin{equation} \label{sol-m3}
   0 \m=\m \widetilde{v}_x(1,t) \m=\m v_x(1,t)-m'(1)^\top w(t)
   \m=\m \left[ m_w^{\top}-m'(1)^\top \right] w(t).
\end{equation}

Recall from Sect.~1 that $D(C_e)=H^2(0,1)$. By \dref{trans-z-u} and
\dref{tu-u-trans}, for any classical solution of the closed-loop
system, the output tracking error is, for every $t\geq 0$,
\begin{equation} \label{sol-m4}
   \begin{array}{l}\disp e_y(t) \m=\m y(t)-r(t) \m=\m C_e[z(\cdot,t)]
   - p_r^\top w(t) \m=\m C_e\Fscr^{-1}[v(\cdot,t)] - p_r^\top w(t)\crr
   \disp\hspace{0.9cm} =\m C_e\Fscr^{-1}[\widetilde{v}(\cdot,t)] +
   \left( C_e\Fscr^{-1}[m] -p_r^\top\right) w(t). \end{array}
\end{equation}
It follows from \dref{sol-m1}-\dref{sol-m4} that if the function $m$
satisfies the following {\em regulator equations}:
\begin{equation} \label{m-equ}
   \left\{\begin{array}{l}\disp im''(x)^\top + m(x)^{\top}S+c_s m(x)
   ^\top \m=\m \Fscr[g](x)p_1^{\top}-k(x,0)p_2^{\top},\crr\disp
   m'(0)^\top \m=\m p_2^{\top},\qquad C_e\Fscr^{-1}[m] \m=\m
   p_r^{\top}, \end{array}\right.
\end{equation}
and we choose $m_w$ in \dref{state-con} so that $m_w=m'(1)$, provided
that the equation \dref{m-equ} is solvable, then the system
\dref{Sch-state-close-z-u} is reduced to \dref{ORSch-tu}, and the
output tracking error for classical solutions of the closed-loop
system becomes, according to \dref{sol-m4},
\begin{equation} \label{sol-me}
   e_y(t) \m=\m y(t)-r(t) \m=\m C_e\Fscr^{-1}[\widetilde{v}
   (\cdot,t)] \m.
\end{equation}

\begin{remark} \label{what_is_F} {\rm
The state feedback operator from \dref{state-con} can be written
in the form \vspace{-2mm}
\begin{equation} \label{F}
   F \m=\m k(1,1)C_m + \Kscr \m,\vspace{-1mm}
\end{equation}
where $\Kscr$ is a bounded linear functional on $\hline$. This shows
(using Proposition \ref{keso_van}) that $F$ is an admissible
observation operator for the semigroups generated by $A$ and $A_h$.
We have from \dref{k-pde} \vspace{-2mm}
$$ k(1,1) \m=\m -\frac{i}{2}\int_0^1 h(\xi) \dd\xi
   -i \left[ \frac{c_s}{2}+q \right] \m,$$
and clearly $C_m=-iB_r^*$. Thus, we can write \vspace{-2mm}
$$ F \m=\m -\left[\frac{c_s}{2}+q+\int_0^1 h(\xi) \dd\xi \right]
   B_r^* + \Kscr \m,$$
which shows that the dominant component of this feedback is
collocated.}
\end{remark}

\begin{remark} \label{Kurdistan} {\rm
The compatible system node $\Sigma=(A_h,[g(\cdot)\ B_l\ B_r],\sbm{C_e
\\ C_m}\nm,0)$ represents the systems \dref{Sch-ab} and also
\dref{Sch}, see Corollary \ref{ProAB}. This is a regular linear
system, according to Proposition \ref{keso_van}. Since $F$ satisfies
\dref{F}, it follows that also the system node $\Sigma_f=(A_h,[g
(\cdot)\ B_l\ B_r],\sbm{C_e \\ C_m\\ F}\nm,0)$ is regular (with input
and output space $\cline^3$). This $\Sigma_f$ has been obtained by
adding a third output to $\Sigma$, namely, $u_f(t)=F_\L z(t)$ (for
classical solutions we may write $u_f(t)=F z(t)$). Now the state
feedback law \dref{state-con} can be written in the abstract output
feedback form that fits Proposition \dref{Ruth}:
\vspace{-1mm}
\begin{equation} \label{Kf}
   \bbm{d_1(t)\\ d_2(t)\\ u(t)} \m= \Kf \bbm{y(t)\\ y_m(t)\\
   u_f(t)} + \rho(t)\m,\   \mbox{ where }\ \  \Kf \m= \bbm{ 0 &
   0 & 0\\ 0 & 0 & 0\\ 0 & 0 & 1} \m,\  \rho(t) \m= \bbm{d_1(t)
   \\ d_2(t)\\ m_w^\top w(t)} \m,
\end{equation}
and $\rho$ is the new input signal of the closed-loop system.}
\end{remark}

\begin{proposition} \label{Simy}
With the notation of Remark {\rm\ref{Kurdistan}}, define \m
$A_{cl}:D(A_{cl})\rarrow\hline$ as follows: \vspace{-2mm}
$$ A_{cl} \m=\m A_h+B_r F_\L \m,\qquad D(A_{cl}) \m=\m \left\{x\in
   H^2(0,1)\ |\ A_hx+B_r F_\L x\in X \m\right\} \m.\vspace{-1mm}$$

The closed-loop system \m $\Sigma_{cl}$ obtained from \m $\Sigma$ with
the feedback law {\rm\dref{state-con}} (described by the equations
{\rm \dref{A+BF}} and the second line of {\rm\dref{Sch-ab}}) is a
regular linear system $\Sigma_{cl}$ with semigroup generator $A_{cl}$,
control operator $B=[g\ B_l\ B_r]$, observation operator $C=\sbm{C_e\\
C_m}$ (restricted to $D(A_{cl})$) and its feedthrough operator $D$ is
the same as for the open-loop system \m $\Sigma$.
\end{proposition}

\begin{proof}
We know from Proposition \ref{keso_van} that the compatible system
node $(A_h,[g(\cdot)\ B_l\ B_r],\sbm{C_{e\L}\\ C_m}\nm,0)$ is
well-posed and strictly proper. Since $F$ satisfies \dref{F}, it
follows that also \vspace{-2mm}
$$ \Sigma_{f0} \m=\m \left( A_h,[g(\cdot)\ B_l\ B_r],\sbm{C_{e\L} \\
   C_m\\ F}\nm,0 \right)$$
is well-posed and strictly proper. This regular system node differs
from $\Sigma_f$ in Remark \ref{Kurdistan} only in its feedthrough
operator: the feedthrough operator of $\Sigma_{f0}$ is zero,
while for \m $\Sigma_f$ it is of the form
$$ D_0 \m=\m \bbm{ 0 & D_1 & D_2\\ 0 & 0 & 0\\ 0 & 0 & 0} \m,\ \
   \mbox{ where }\ \ \ \ \begin{array}{l} D_1 \m=\m \lim_{\l\rarrow
   \infty,\m\l\in\rline} C_e(\l I-A_h)^{-1} B_l \m,\crr D_2 \m=\m
   \lim_{\l\rarrow\infty,\m\l\in\rline} C_e(\l I-A_h)^{-1} B_r \m.
   \end{array}$$
The limits $D_1$ and $D_2$ could be any numbers in $\cline$, because
$C_e$ has not been specified. According to the last part of
Proposition \ref{Savta}, the system \m $\Sigma_f$ is equivalent to
$$ \Sigma_f \m=\m \left( A_h,[g(\cdot)\ B_l\ B_r],\sbm{C_{e\L} \\ C_m
   \\ F}\nm,D_0 \right)$$
in the sense that these systems have the same system operator and the
same transfer function. (According to the theory of system nodes,
having the same system operator means that they are the same system.)
We denote by $\GGG$ and $\GGG_0$ the transfer functions of \m
$\Sigma_f$ and $\Sigma_{f0}$ respectively, so that $\GGG(s)=\GGG_0
(s)+D_0$. We see that $I-\Kf\GGG(s)$ has a uniformly bounded inverse
on some right half-plane, because $\Kf\GGG(s)=\Kf\GGG_0(s)$ and
$\GGG_0$ is strictly proper. Note that $\Kf D_0=0$. Thus, we can 
apply Proposition \ref{Ruth} to conclude that \m $\Sigma_f$ with the
feedback law \dref{state-con}, which is equivalent to \dref{Kf},
leads to a well-posed and regular closed-loop system $\Sigma_{cl,f}$.

According to Proposition \ref{Ruth}, after a little computation, we
find that \vspace{-2mm}
$$ \Sigma_{cl,f} \m=\m \left( A_{cl},[g(\cdot)\ B_l\ B_r],\sbm{
   C_{e\L}+D_2 F\\ C_m\\ F}\nm,D_0 \right) \m,$$
where $A_{cl}$ is defined in the proposition. Another short
computation shows that the above system node $\Sigma_{cl,f}$ is
equivalent to \vspace{-2mm}
$$ \Sigma_{cl,f} \m=\m \left( A_{cl},[g(\cdot)\ B_l\ B_r],\sbm{
   C_e\\ C_m\\ F}\nm,0 \right) \m.$$
If we ignore the third output of this system, $u_f$ introduced in
Remark \ref{Kurdistan}, then we obtain the closed-loop system
$\Sigma_{cl}$ stated in the proposition. We remark that $D$ consists
of the first two lines of $D_0$ and that the restrictions of
$C_{e\L}+D_2 F$ and of $C_e$ to $D(A_{cl})$ are equal. \end{proof}

\begin{proposition} \label{utolso}
We use the notation of Proposition {\rm\ref{Simy}}. Assume that the
regulator equations {\rm\dref{m-equ}} have a solution $m$ and $m_w=
m'(1)$, so that {\rm\dref{ORSch-tu}} and {\rm\dref{sol-me}} hold.

Then $C_e\Fscr^{-1}$ is an admissible observation operator for the
group generated by $\aline$ from {\rm\dref{Alina}}.
\end{proposition}

\begin{proof}
Consider the cascade connection of the closed-loop system
$\Sigma_{cl}$ with the exosystem from \dref{dis-produce} according to
\dref{A+BF}, so that all three inputs of $\Sigma_{cl}$ come from the
finite-dimensional exosystem. Since $\Sigma_{cl}$ is well-posed, it
follows that this cascade connection is again well-posed, implying
that for any $T>0$ there exists an $m_T>0$ such that \vspace{-2mm}
$$ \int_0^T \| y(t) \|^2 \m\leq\m m_T \left\| \bbm{z(\cdot,0)\\ w(0)}
   \right\|^2 \m.$$
Clearly a similar estimate holds for the signal $r$, and using
\dref{sol-me} it follows that a similar estimate holds for $e_y$: for
some $\tilde m_T>0$, \vspace{-2mm}
$$ \int_0^T \| e_y(t) \|^2 \m\leq\m \tilde m_T \left\| \bbm{z(\cdot,0)
   \\ w(0)} \right\|^2 \m.$$
Now consider the special case $w(0)=0$. Then according to
\dref{trans-z-u} and \dref{tu-u-trans}, we have $\widetilde v=\Fscr z$
and according to \dref{ORSch-tu} and \dref{sol-me} we have \m $\dot
{\widetilde v}(t)=(\aline-c_s I)\widetilde v(t)$ and $e_y(t)=C_e
\Fscr^{-1}\widetilde v(t)$. From
$$ \int_0^T \| e_y(t) \|^2 \m\leq\m \tilde m_T \|z(0)\|^2
   \m\leq\m \tilde m_T \|\Fscr^{-1}\|^2 \|\widetilde v(0)\|^2 \m.$$
This shows that $C_e\Fscr^{-1}$ is an admissible observation
operator for the group generated by $\aline-c_s I$ (equivalently, for
the group generated by $\aline$).
\end{proof}

\begin{remark} {\rm Let $\widetilde v$ satisfy \dref {ORSch-tu} and
denote $\widehat{v}(\cdot,t):=\Fscr^{-1}[\widetilde{v}(\cdot,t)]$.
Then $\widehat{v}(x,t)$ is governed by
$$ \left\{\begin{array}{l}\disp \widehat{v}_{t}(x,t) \m=\m -i
   \widehat{v}_{xx}(x,t)+h(x)\widehat{v}(x,t),  \crr\disp
   \widehat{v}_x(0,t)=-iq\widehat{v}(0,t), \qquad\crr\disp
   \widehat{v}_x(1,t) \m=\m K(1,1)\widetilde{v}(1,t)+\int_0^1
   K_x(1,\xi)\widetilde{v}(\xi,t)\dd\xi = K(1,1)\bigg[\widehat{v}
   (1,t)\crr\disp\m\m \ \ \ -\int_0^1k(1,\xi)\widehat{v}
   (\xi,t)\dd\xi\bigg] + \int_0^1K_x(1,\xi)\Big(\widehat{v}(\xi,t)
   -\int_0^1k(\xi,\zeta)\widehat{v}(\zeta,t)\dd\zeta\Big))\dd\xi \m.
   \end{array}\right.$$ }
\end{remark}

We state a lemma which describes the solvability condition of the
regulator equation \dref{m-equ}. This lemma is related to
\cite[Theorem 5.2]{Viv2014tac}.

\begin{lemma}
The regulator equation {\rm\dref{m-equ}} has a unique solution if and
only if \m $C_e\Fscr^{-1}[\cosh(\sqrt{-i(\l+c_s)}\cdot)]\neq 0$, for
all $\l\in\sigma(S)$.
\end{lemma}

\begin{proof}
Since $S$ is diagonalizable, there exists a square matrix 
$$V \m=\m [v_1,v_2,\ldots v_{n_w}] \m,\ \ \ v_j\in\rline^{n_w} \m,$$
such that $V^{-1}SV={\rm diag}(\l_1,\l_2,\ldots\l_{n_w})$, where 
$\l_j$, $j=1,2,\ldots n_w$ are the eigenvalues of $S$. Multiply with 
$v_j$ from the right in \dref{m-equ} to obtain
\begin{equation} \label{m-equ-vi}
   \left\{\begin{array}{l}\disp \bar{m}''_j(x)-i\l_j\bar{m}_j(x)-ic_s
   \bar{m}_j(x) \m=\m -i[\Fscr[g](x)p_1^{\top}v_j-k(x,0)p_2^{\top}v_j],
   \crr\disp \bar{m}'_j(0) \m=\m p_2^{\top}v_j,\qquad C_e\Fscr^{-1}
   [\bar{m}_j]=p_r^{\top}v_j,\quad j=1,2,\ldots n_w, \end{array}\right.
\end{equation}
where $\bar{m}_j=m(x)^{\top}v_j$, $j=1,2,\ldots,n_w$. If $\l_j+c_s\neq
0$, the general solution of the first equation of \dref{m-equ-vi} is
of the following form (with the coefficients $\gamma_1,\gamma_2$ to be
determined):
$$ \begin{array}{l}\disp \bar{m}_j(x) \m=\m \gamma_1\cosh(\sqrt{-i
   (\l_j+c_s)}x)+\gamma_2\frac{\sinh(\sqrt{-i(\l_j+c_s)}x)}{\sqrt{-i
   (\l_j+c_s)}}\\ \disp\hspace{1.3cm} -i\int_0^x \big[\Fscr[g](\xi)
   p_1^{\top}v_j-k(\xi,0)p_2^{\top}v_j\big] \frac{\sinh(\sqrt{-i(\l_j
   +c_s)}(x-\xi))}{\sqrt{-i(\l_j+c_s)}} \m\dd\xi \m. \end{array}$$
Substituting this into the boundary conditions in \dref{m-equ-vi}, we
get \vspace{-1mm}
\begin{equation} \label{ga1ga2-coef}
   \left\{\begin{array}{l}\disp \gamma_2=p_2^{\top}v_j,\\ \disp
   \gamma_1 C_e\Fscr^{-1}[\cosh(\sqrt{-i(\l_j+c_s)}\cdot)] +\gamma_2
   C_e\Fscr^{-1}\bigg[\frac{\sinh(\sqrt{-i(\l_i+c_s)}x)}{\sqrt{-i
   (\l_j+c_s)}}\bigg]\\ \disp\hspace{0.3cm} + C_e\Fscr^{-1} \bigg[-i
   \int_0^\cdot \big[\Fscr[g](\xi)p_1^{\top}v_i-k(\xi,0)p_2^{\top}v_j
   \big]\\ \hspace{4cm} \disp\times\frac{\sinh(\sqrt{-i(\l_j+c_s)}
   (\cdot-\xi))}{\sqrt{-i(\l_j+c_s)}} \dd\xi\bigg] \m=\m p_r^{\top} 
   v_j. \end{array}\right.
\end{equation}
It is obvious that the coefficients $\gamma_1$, $\gamma_2$ can be
uniquely determined by equation \dref{ga1ga2-coef} if and only if
$C_e\Fscr^{-1}[\cosh(\sqrt{-i(\l_j+c_s)}\cdot)]\neq 0$.

If $\l_j+c_s=0$, then the solutions of the first equation in
\dref{m-equ-vi} are of the form
\begin{equation} \label{mi-gens-0}
   \bar{m}_j(x) \m=\m \gamma_1+\gamma_2x-i\int_0^x(x-\xi)[\Fscr[g]
   (\xi)p_1^{\top}v_j-k(\xi,0)p_2^{\top}v_j]\dd\xi,
\end{equation}
where $\gamma_1$, $\gamma_2$ are the coefficients to be determined.
Substituting \dref{mi-gens-0} into the boundary conditions in
\dref{m-equ-vi}, we get $\gamma_2=p_2^{\top}v_i$ and, denoting by
$\eta$ the identity function, $\eta(x)=x$, \vspace{-2mm}
\begin{equation}\begin{array}{l}\disp
   \gamma_1 C_e\Fscr^{-1}[1] = -\gamma_2 C_e\Fscr^{-1}[\eta] +
   p_r^{\top}v_j\crr\disp\hspace{1.5cm}- C_e
   \Fscr^{-1} \bigg[-i\int_0^\cdot (\cdot-\xi)[\Fscr[g](\xi)
   p_1^{\top}v_j-k(\xi,0)p_2^{\top}v_j]\dd\xi\bigg].\end{array}
\end{equation}
It is clear that $\gamma_1$ can be uniquely determined from this
equation if and only if $C_e\Fscr^{-1}[1]\neq 0$.
\end{proof}

Now, with the state feedback, we turn to the closed-loop system which
is composed of \dref{Sch}, \dref{dis-produce}, \dref{state-con} and
\dref{sol-me}, that is
\begin{equation}\label{Sch-state-closed}
   \left\{\begin{array}{l} z_t(x,t) \m=\m -iz_{xx}(x,t)+h(x)z(x,t)+
   g(x)p^\top_1 w(t),\crr \disp z_x(0,t) \m=\m -iqz(0,t)+p^\top_2
   w(t),\crr\disp z_x(1,t) \m=\m k(1,1) z(1,t) + \int_0^1 k_x(1,\xi)
   z(\xi,t)\dd\xi + m_w^\top w(t),\crr\disp z(\cdot,0) \m=\m z_0
   (\cdot)\in L^2[0,1],\crr\disp \dot{w}(t) \m=\m Sw(t),\qquad w(0)=
   w_0\in\rline^{n_w},\crr e_y(t) \m=\m y(t)-r(t) \m=\m C_e
   [z(\cdot,t)] - p_r(t)^\top w(t). \end{array} \right.
\end{equation}

The following is the main result of this section.

\begin{theorem} \label{Thm-statere}
Let $c_s>0$ and let the functions $k$ and $m$ be solutions of
{\rm\dref{k-pde}} and {\rm\dref{m-equ}}. Suppose that \vspace{-2mm}
$$ C_e\Fscr^{-1}[\cosh(\sqrt{-i(\l+c_s)}\cdot)] \m\neq\m 0 \FORALL \l
   \in\sigma(S) \m.$$

Then the state feedback law {\rm\dref{state-con}} with $m_w^{\top}
=m'(1)^{\top}$ solves the output regulation problem for the system
{\rm\dref{Sch-state-closed}}, i.e., $e_y\in L_{\alpha}[0,\infty)$ for
some $\alpha<0$. If $C_e$ is bounded, then there exist $M,\mu>0$ such
that $|e_y(t)|\leq Me^{-\mu t}$ holds for all $t\geq 0$.
\end{theorem}

\begin{proof}
We have seen after \dref{ORSch-tu} that \m $\dot{\widetilde{v}}
(\cdot,t)=(\aline-c_s I)\widetilde{v}(\cdot,t)$, where $\aline$ is
skew-adjoint. Clearly $\aline-c_s I$ generates an exponential stable
operator group, which, jointly with the admissibility of the
observation operator $C_e\Fscr^{-1}$ (see Proposition \ref{utolso})
implies that $e_y=C_{e\L}\Fscr^{-1}[\widetilde{v}]\in L^2_{\alpha}[0,
\infty)$ with $\alpha\in(-c_s,0)$, see \cite[Proposition 4.3.6]
{obs_book}. If the observation operator $C_e$ is bounded, then by the
boundedness of the transformation $\Fscr^{-1}$, there exist three
constants $C_0,M,\mu>0$ such that\\ \m\ \ \ \ \ \ $|e_y(t)|=|C_e\Fscr
^{-1}[\widetilde{v}](\cdot,t)|\leq C_0\|\widetilde{v}(\cdot,t)\|\leq
C_0 Me^{-\mu t}\|\widetilde{v}(\cdot,0)\|$. \end{proof}

\section{Observer design} 

\ \ \ The full states $w(t)$ and $z(\cdot,t)$ used in \dref{state-con}
are not always available (as measurements) to the controller. Thus, to
implement the feedback law \dref{state-con}, we need to design an
observer for the combined system \dref{Sch} and \dref{dis-produce}, to
recover its state from the output measurement $y_m(t)=z(1,t)$ and from
the reference $r(t)$. Since $(q_r^{\top},S_r)$ is observable, there
exists an observer gain $l_r\in\rline^{n_r}$ such that
$S_r+l_rq_r^{\top}$ is Hurwitz. So, we can use the finite
dimensional reference observer
\begin{equation} \label{wr-obser}
   \dot{\widehat{w}}_r(t) \m=\m S_r\widehat{w}_r(t)+l_r(q_r^{\top}
   \widehat{w}_r(t)-r(t)),
\end{equation}
where $\widehat{w}_r(t)$ is the estimate of $w_r(t)$ in
\dref{dis-produce}. In order to estimate $z(\cdot,t)$ and $w_d$ in
\dref{Sch} and \dref{dis-produce}, we design the following observer:
\begin{equation} \label{Sch-obser-r}
   \left\{\begin{array}{l} \dot{\widehat{w}}_d(t) \m=\m S_d\widehat
   {w}_d(t) + l_d(\widehat{z}(1,t)-y_m(t)), \crr\disp \widehat{z}_t
   (x,t) \m=\m -i\widehat{z}_{xx}(x,t)+h(x)\widehat{z}(x,t)+g(x)
   q^{\top}_{d_1}\widehat{w}_d(t) +l(x)\left[\widehat{z}(1,t)-y_m
   (t)\right],\crr\disp \widehat{z}_x(0,t) \m=\m -iq\widehat{z}(0,t)
   +q^\top_{d_2}\widehat{w}_d(t), \crr\disp \widehat{z}_x(1,t) \m=\m
   u(t)+l_0(\widehat{z}(1,t)-y_m(t)),\end{array}\right.
\end{equation}
where $l(\cdot)$, $l_0$ are observer gains, to be designed later. It
should be noted that the above observer \dref{Sch-obser-r} is
implemented based on the boundary measurement $y_m(t)$ and the input
signal $u(t)$. Let
\begin{equation*} 
   \widetilde{w}_r(t) \m=\m \widehat{w}_r(t)-w_r(t),\quad
   \widetilde{w}_d(t) \m=\m \widehat{w}_d(t)-w_d(t),\quad
   \widetilde{z}(x,t) \m=\m \widehat{z}(x,t)-z(x,t)
\end{equation*}
be the observer errors. Then, by \dref{Sch}, \dref{wr-obser} and
\dref{Sch-obser-r}, $\widetilde{w}_d(t)$, $\widetilde{w}_r(t)$ and
$\widetilde{z}(x,t)$ satisfy
\begin{equation} \label{Sch-obser-err-z}
   \left\{\begin{array}{l} \dot{\widetilde{w}}_r(t) \m=\m (S_r+l_r
   q_r^{\top})\widetilde{w}_r(t),\qquad \dot{\widetilde{w}}_d(t)
   \m=\m S_d\widetilde{w}_d(t)+l_d\widetilde{z}(1,t),\crr\disp
   \widetilde{z}_t(x,t) \m=\m -i\widetilde{z}_{xx}(x,t)+h(x)
   \widetilde{z}(x,t)+g(x)q^{\top}_{d_1}\widetilde{w}_d(t)+l(x)
   \widetilde{z}(1,t),\crr\disp \widetilde{z}_x(0,t) \m=\m -iq
   \widetilde{z}(0,t)+q^{\top}_{d_2}\widetilde{w}_d(t),\qquad
   \widetilde{z}_x(1,t) \m=\m l_0\widetilde{z}(1,t),
   \end{array}\right.
\end{equation}
which has to be exponentially stabilized. In order to find the
observer gains $l(\cdot)$, $l_0$ that ensure that
\dref{Sch-obser-err-z} is exponentially stable, we look for the
{\em backstepping transformation} \vspace{-2mm}
\begin{equation} \label{trans-tz-u}
   \widetilde{z}(x,t) \m=\m \Fscr_o[e](x,t) :=\m e(x,t)-\int_x^1
   p(x,\xi)e(\xi,t)\dd\xi,
\end{equation}
that transforms \dref{Sch-obser-err-z} into the following system:
\begin{equation} \label{Sch-obser-err-r}
   \left\{\begin{array}{l} \dot{\widetilde{w}}_r(t) \m=\m (S_r+l_r q_r
   ^{\top})\widetilde{w}_r(t),\qquad \dot{\widetilde{w}}_d \m=\m S_d
   \m\widetilde{w}_d(t)+l_de(1,t), \crr\disp e_t(x,t) \m=\m -ie_{xx}
   (x,t)-c_oe(x,t)+\widetilde{g}(x)^{\top}\widetilde{w}_d(t) +
   \widetilde{l}(x)e(1,t),\;x\in (0,1), \crr\disp e_x(0,t) \m=\m
   q^{\top}_{d_2}\widetilde{w}_d(t), \qquad e_x(1,t)=0, \crr\disp
   e(x,0) \m=\m e_0(x) \m=\m \Fscr_o^{-1}[\widetilde{z}_0](x),
   \end{array}\right.
\end{equation}
where $\widetilde{g}(x)^{\top}$ is given by
$\widetilde{g}(x)^{\top}=\mathcal{F}_o^{-1}[g](x)q^{\top}_{d_1}$ and
$\widetilde{l}(x)$ is needed as an additional degree of freedom for
the subsequent design.

By the third equations of \dref{Sch-obser-err-z} and
\dref{Sch-obser-err-r}, and the transformation \dref{trans-tz-u},
through integration by parts we obtain
\begin{eqnarray} \label{p-eq1}
  &&g(x)q^{\top}_{d_1}\widetilde{w}_d(t)+l(x)e(1,t) \m=\m g(x)q^\top
   _{d_1}\widetilde{w}_d(t)+l(x)\widetilde{z}(1,t)\crr\disp
    &&=\widetilde{z}_t(x,t)+i\widetilde{z}_{xx}(x,t)-h(x)\widetilde{z}
   (x,t)\crr \disp
  &&= e_t(x,t)-\int_x^1p(x,\xi)e_t(xi,t)\dd\xi+i\bigg[e(x,t)-\int_x^1
   p(x,\xi)e(\xi,t)\dd\xi\bigg]_{xx}\crr\disp
  &&\hspace{0.4cm}-h(x)\bigg[e(x,t)-\int_x^1 p(x,\xi)e(\xi,t)\dd\xi
   \bigg]\crr\disp
  &&=-ie_{xx}(x,t)-c_oe(x,t)+\widetilde{g}(x)^{\top}\widetilde{w}_d
   (t)+\widetilde{l}(x)e(1,t)\crr\disp
  &&\hspace{0.4cm}-h(x)\bigg[e(x,t)-\int_x^1p(x,\xi)e(\xi,t)\dd\xi
   \bigg]+\int_x^1 p(x,\xi)c_o e(\xi,t)\dd\xi\crr\hspace{0.4cm}\disp
   &&\hspace{0.4cm}+i\int_x^1 p_{\xi\xi}
   (x,\xi)e(\xi,t)\dd\xi  -i\int_x^1 p_{xx}(x,\xi)e(\xi,t)\dd\xi\crr
   \hspace{0.4cm}
  &&\hspace{0.4cm}-\int_x^1 p(x,\xi)\widetilde{g}(\xi)^{\top}\dd\xi
   \widetilde{w}_d(t)-\int_x^1 p(x,\xi)\widetilde{l}(x)\dd\xi e(1,t)
  \crr\hspace{0.4cm} \disp
  &&\hspace{0.4cm}+i\bigg[e_{xx}(x,t)+\frac{d}{dx}p(x,t)e(x,t)+p(x,x)
   e_x(x,t)+p_x(x,x)e(x,t)\bigg]\crr\disp
  &&\hspace{0.4cm}+i[p(x,1)e_x(1,t)-p(x,x)e_x(x,t)-p_\xi(x,1)e(1,t)+
   p_\xi(x,x)e(x,t)]\crr\disp
  &&=\bigg[2i\frac{d}{dx}p(x,t)e(x,t)-h(x)-c_o\bigg]e(x,t)\crr\disp
  &&\hspace{0.4cm}+[\Fscr_o[\widetilde{l}(x)]-ip_\xi(x,1)]e(1,t) + 
   \Fscr_o[\widetilde{g}(x)^\top]
   \widetilde{w}_d(t)\crr\hspace{0.4cm}\disp
  &&\hspace{0.4cm}+\int_x^1[-ip_{xx}(x,\xi)+ip_{\xi\xi}(x,\xi)+h(x)
   p(x,\xi)+c_o p(x,\xi)]e(\xi,t)\dd\xi \m.
\end{eqnarray}
By the fourth equations of \dref{Sch-obser-err-z} and
\dref{Sch-obser-err-r}, and the transformation \dref{trans-tz-u}, we
obtain
$$ \begin{array}{l}\disp 0 \m=\m e_x(0,t)-q^{\top}_{d_2}\widetilde
   {w}_d(t) \m=\m \widetilde{z}_x(0,t)-p(0,0) e(0,t) + \int_0^1 p_\xi
   (0,\xi)e(\xi,t)\dd\xi\\ \m\hspace{60mm} -q^\top_{d_2}
   \widetilde{w}_d(t)\crr\hspace{0.3cm}
   \disp =\m -iq[e(0,t)-\int_x^1 p(0,\xi)e(\xi,t)\dd\xi] -
   p(0,0)e(0,t)+\int_0^1 p_\xi(0,\xi)e(\xi,t)\dd\xi\crr\hspace{0.3cm}
   \disp =\m -[p(0,0)+qi]e(0,t)+\int_0^1[p_\xi(0,\xi)+qi p(0,\xi)]
   e(\xi,t)\dd\xi.\end{array}$$
By the fifth equations of \dref{Sch-obser-err-z} and
\dref{Sch-obser-err-r} and the transformation \dref{trans-tz-u},
\begin{equation} \label{p-eq3}
   0 \m=\m e_x(1,t) \m=\m \widetilde{z}_x(1,t)-p(1,1)e(1,t) \m=\m
   l_0\widetilde{z}(1,t)-p(1,1)e(1,t) \vspace{-1mm}
\end{equation}
$$ \m\hspace{63mm} =\m [l_0-p(1,1)]e(1,t) \m.$$
It follows from \dref{p-eq1}-\dref{p-eq3} that the kernel function
$p(x,\xi)$ in \dref{trans-tz-u} should satisfy
$$ \left\{ \begin{array}{l} p_{\xi\xi}(x,\xi)-p_{xx}(x,\xi) \m=\m
   (h(x)+c_o)ip(x,\xi),\; c_o>0,\crr\disp p_x(0,\xi)+qi p(0,\xi)
   \m=\m 0,\crr\disp p(x,x) \m=\m -\frac{i}{2}\int_0^x (h(\xi)+c_o)
   \dd\xi-qi, \end{array}\right.$$
and that we should choose the observer gains $l(\cdot)$ and $l_0$ in
\dref{Sch-obser-r} so that \vspace{-1mm}
$$ l(x) \m=\m \Fscr_o[\widetilde{l}](x)-ip_\xi(x,1),\qquad
   l_0 \m=\m p(1,1) \m.$$
By \cite[Theorem 2.2]{AS-Krstic-book}, the above equations in $p$ have
a unique solution $p\in C^2(\Om)$. We note that we have still not
obtained the final expression of the observer gain $l(\cdot)$, because
$\widetilde{l}(\cdot)$ is a new design function. In order to find
$\widetilde{l}(\cdot)$ in \dref{Sch-obser-err-r} so that the
``$e$-part'' of the system \dref{Sch-obser-err-r} is exponentially
stable in $L^2[0,1]$, we further introduce the {\em error
transformation} \vspace{-1mm}
\begin{equation} \label{tu-u-trans-nn}
   \widetilde{e}(x,t) \m=\m e(x,t)-n(x)^{\top}\widetilde{w}_d(t).
\end{equation}
It is expected that under the above transformation, the system
\dref{Sch-obser-err-r} can be transformed into
\begin{equation} \label{te-equ}
   \left\{\begin{array}{l} \dot{\widetilde{w}}_r(t) \m=\m (S_r+l_r
   q_r^{\top})\widetilde{w}_r(t),\crr\disp \dot{\widetilde{w}}_d \m=\m
   (S_d+l_dn(1)^{\top})\widetilde{w}_d(t)+l_d\widetilde{e}(1,t),
   \crr\disp \widetilde{e}_t(x,t) \m=\m -i\widetilde{e}_{xx}(x,t)-c_o
   \widetilde{e}(x,t), \crr\disp \widetilde{e}_x(0,t)=0, \qquad
   \widetilde{e}_x(1,t) \m=\m 0.\end{array}\right.
\end{equation}
Substituting \dref{tu-u-trans-nn} into the third equation of
\dref{te-equ}, we derive
\begin{equation} \label{sol-n1}
   \begin{array}{l} 0 \m=\m \widetilde{e}_t(x,t)+i\widetilde{e}_{xx}
   (x,t)+c_o\widetilde{e}(x,t)\crr\disp =\m e_t(x,t)-n(x)^{\top}S_r
   \widetilde{w}_d(t)-n(x)^{\top}l_de(1,t)\crr\disp
   \ \ \   +i[e_{xx}(x,t)-n''(x)^{\top}\widetilde{w}_d(t)]
   +c_o e(x,t)-c_on(x)^{\top}\widetilde{w}_d(t)\crr\disp
   =\m [\widetilde{l}(x)-n(x)^{\top}l_d]e(1,t)+[\widetilde{g}(x)^T -
   n(x)^{\top}S_d-in''(x)^{\top}-c_on(x)^{\top}]\widetilde{w}_d(t)
   \m=\m 0.\end{array}
\end{equation}
Substituting \dref{tu-u-trans-nn} into the fourth equation of
\dref{te-equ}, we have
\begin{equation} \label{sol-n2}
   0 \m=\m \widetilde{e}_x(0,t)=e_x(0,t)-n'(0)\widetilde{w}_d(t)
   \m=\m [q^{\top}_{d_2}-n'(0)^{\top}]\widetilde{w}_d(t) \m=\m 0.
\end{equation}
Substituting \dref{tu-u-trans-nn} into the fifth equation of
\dref{te-equ}, we obtain
\begin{equation} \label{sol-n3}
   0 \m=\m \widetilde{e}_x(1,t)=e_x(0,t)-n'(1)^{\top}\widetilde{w}
   _d(t) \m=\m -n'(1)^{\top}\widetilde{w}_d(t) \m=\m 0.
\end{equation}
It follows from \dref{sol-n1}-\dref{sol-n3} that $n(\cdot)$ must
satisfy the following equations:
\begin{equation} \label{n-equ}
   \left\{\begin{array}{l}\disp in''(x)^{\top}+n(x)^{\top}S_d+c_on(x)
   ^\top =\m \widetilde{g}(x)^{\top},\crr\disp n'(0)^{\top} =\m
   q^\top_{d_2},\qquad n'(1)^{\top} =\m 0. \end{array}\right.
\end{equation}
If we choose $\widetilde{l}$ so that $\widetilde{l}(x)=n(x)^\top l_d$,
provided that the equation \dref{n-equ} is solvable, then the system
\dref{Sch-obser-err-r} becomes \dref{te-equ}. Thus, the observer gains
$l(\cdot)$ and $l_0$ in \dref{Sch-obser-r} are designed as follows:
\begin{equation} \label{lxl0-expr-fin}
   l(x) \m=\m \Fscr_o(n(x)^\top l_d)-ip_\xi(x,1),\quad l_0 \m=\m
   p(1,1) \m,
\end{equation}
provided that the equation \dref{n-equ}  has a solution.

\begin{lemma}
The equations {\rm\dref{n-equ}} have a unique solution if and only if
$\sigma_o\cap\sigma(S_d)=\emptyset$, where $\sigma_o=\{-j^2\pi^2i-c_o
\}$ is the eigenvalue set of the ``$\widetilde{e}$-part'' of
{\rm\dref{te-equ}}.
\end{lemma}

\begin{proof}
Since $S$ is diagonalizable, there exists a matrix $V=[v_1,v_2,\ldots
v_{n_d}]$, $v_j\in\rline^{n_w}$, $j=1,2,\ldots n_d$, such that
$V^{-1}SV={\rm diag}(\l_1,\l_2,\dots\l_{n_d})$, where $\l_j$, $j=1,2,
\ldots n_d$ are the eigenvalues of $S_d$. Multiply by $v_j$ from the
right in \dref{m-equ} to obtain
\begin{equation} \label{n-equ-vi}
   \left\{\begin{array}{l}\disp i\bar{n}''_j(x)+\l_j\bar{n}_j(x)+c_o
   \bar{n}_j(x) \m=\m \widetilde{g}(x)^{\top}v_j,\crr\disp \bar{n}'_j
   (0) \m=\m q^{\top}_{d_2}v_j,\qquad \bar{n}'_j(1) \m=\m 0,\ j=1,2,
   \ldots n_d,\end{array}\right.
\end{equation}
where $\bar{n}_j=n(x)^{\top}v_j$, $j=1,2,\ldots,n_d$. If $\l_j+c_o
\neq 0$, the solutions of the first equation in \dref{n-equ-vi} are of
the form \vspace{-2mm}
\begin{equation} \label{ni-gens}
   \begin{array}{l}\disp \bar{n}_j(x) \m=\m \gamma_1\cosh\left(\sqrt
   {-i(\l_j+c_o s)}x\right) + \gamma_2\frac{\sinh(\sqrt{-i(\l_j+c_o)}
   x)}{\sqrt{-i(\l_j+c_o)}}\\ \disp\hspace{1.3cm} +\int_0^x \big[-i
   \widetilde{g}(\xi)^{\top}v_j\big] \frac{\sinh(\sqrt{-i(\l_j+c_o)}
   (x-\xi))}{\sqrt{-i(\l_j+c_o)}}\dd\xi,\end{array}
\end{equation}
where $\gamma_1,\gamma_2$ are coefficients to be determined.
Substituting \dref{ni-gens} into the boundary conditions in
\dref{n-equ-vi} yields \vspace{-1mm}
$$ \left\{ \begin{array}{l} \gamma_2 \m=\m q^{\top}_{d_2}v_j,\crr
   \disp \gamma_1\sqrt{-i(\l_j+c_o)}\sinh\sqrt{-i(\l_j+c_o)}+
   \gamma_2\cosh\sqrt{-i(\l_j+c_o)}\crr\disp\hspace{1cm} \m=\m
   \int_0^1 i\widetilde{g}(\xi)^{\top}v_j\cosh(\sqrt{-i(\l_j+c_o)}
   (1-\xi))\dd\xi. \end{array}\right.$$
It is obvious that the coefficients $\gamma_1,\gamma_2$ are uniquely
determined if and only if $\sinh\sqrt{-i(\lambda_j+c_o)}\neq0$.  It is
easy to see that $\sinh\sqrt{-i(\lambda_j+c_o)}\neq0$ is equivalent to
$\sigma_o\cap\sigma(S_d)=\emptyset$.

If $\l_j+c_o=0$, the solutions of the first equation in
\dref{n-equ-vi} are of the form
\begin{equation} \label{ni-gens-0}
   \bar{n}_j(x) \m=\m \gamma_1+\gamma_2x+\int_0^x(x-\xi)\big[-i
   \widetilde{g}(\xi)^{\top}v_j\big]\dd\xi,
\end{equation}
where $\gamma_1$, $\gamma_2$ are the coefficients to be determined.
Substituting \dref{ni-gens-0} into the boundary conditions in
\dref{n-equ-vi}, we get $\gamma_2=q^{\top}_{d_2}v_j$ and
$\gamma_2=\int_0^1i\widetilde{g}(\xi)^{\top}v_j\dd\xi$.  It is obvious
that the $\gamma_1$ cannot be uniquely determined and that there is no
solution $\gamma_2$ if \m $q^\top_{d_2}v_j\neq\int_0^1 i\widetilde{g}
(\xi)^{\top}v_j\dd\xi$. Therefore, \dref{n-equ} admits a unique
solution if and only if $\sigma_o\cap\sigma(S_d)=\emptyset$.
\end{proof}

The next result confirms the existence, uniqueness and the
exponentially stability of the solutions of the observer error system
\dref{Sch-obser-err-z}. Rewrite the system \dref{te-equ} in the form
$$ \frac{\dd}{\dd t}(\widetilde{w}_r(t),\widetilde{w}_d(t),
   \widetilde{e}(\cdot,t))^\top \m=\m \AAA(\widetilde{w}_r(t),
   \widetilde{w}_d(t),\widetilde{e}(\cdot,t))^\top$$
where the operator $\AAA:D(\AAA)\to $ is defined as follows:
$$ \left\{ \begin{array}{l}\AAA(X_r,X_d,\phi(x)) \crr\disp\m=\m 
   ((S_r+l_r q_r^{\top})X_r,(S_d+l_dn(1)^\top)X_d+l_d\phi(1),-i
   \phi''(x)-c_o\phi(x)),\crr\disp D(\AAA) \m=\m \{(X_r,X_d,\phi(x))
   \in\rline^{n_r}\times\rline^{n_d}\times H^2(0,1)\m|\ \phi'(0)=0,\ 
   \phi'(1)=0\}. \end{array}\right.$$

\begin{theorem} \label{Thm-obser}
Let $\sigma_o\cap\sigma(S_d)=\emptyset$. Suppose that the observer
gains $l(x)$, $l_0$ are given by {\rm\dref{lxl0-expr-fin}} and the
gain $l_r\in\rline^{n_r}$ is chosen so that $S_r+l_rq_r^{\top}$ is
Hurwitz. Suppose that $S_d+l_dn(1)^{\top}$ is also Hurwitz. Moreover,
assume that $S_r+l_rq_r^{\top}$ has simple and stable eigenvalues
$\l_{rj}$ with the corresponding eigenvectors $X_{rj}\in\rline^{n_r}$,
$j=1,2,\ldots n_r$, and $S_d+l_dn(1)^{\top}$ has simple and stable
eigenvalues $\l_{dj}$ with the corresponding eigenvectors $X_{dj}\in
\rline^{n_d}$, $j=1,2,\ldots n_d$, and $\l_{rj_1}\neq\l_{dj_2}$ for
$1\leq j_1\leq n_r$, $1\leq j_2\leq n_d$. Let $c_s$, $c_o>0$. Then
{\rm\dref{wr-obser}} with {\rm\dref{Sch-obser-r}} is an observer for
the system {\rm\dref{Sch}}. Moreover, the observer error dynamics
{\rm\dref{Sch-obser-err-z}} is exponentially stable in the sense that
for some $M\geq 1,\mu>0$, \vspace{-1mm}
\begin{equation} \label{Thm-re-err-expon}
   \|(\widetilde{w}_d(t),\widetilde{w}_r(t),\widetilde{z}(\cdot,t))\|
   \m\leq\m Me^{-\mu t}\|(\widetilde{w}_d(0),\widetilde{w}_r(0),
   \widetilde{z}(\cdot,0))\|.
\end{equation}
\end{theorem}

\begin{proof}
We compute the eigenvalues and the corresponding eigenfunctions of
$\AAA$. We solve $\AAA(X_r,X_d,\phi(x))=\l(X_r,X_d,\phi(x))$, where
$\l\in\sigma(\AAA)$ and $(X_r,X_d,\phi(x))\in D(\AAA)$, to obtain
\vspace{-1mm}
\begin{equation} \label{eig-equ}
   \left\{ \begin{array}{l}(S_r+l_rq_r^{\top})X_r \m=\m \l X_r,\qquad
   (S_d+l_dn(1)^{\top})X_d+l_d\phi(1) \m=\m \l X_d,\crr\disp
   -i\phi''(x)-c_o\phi(x) \m=\m \l\phi(x),\qquad \phi'(0) \m=\m
   0,\qquad \;\phi'(1) \m=\m 0.\end{array}\right.
\end{equation}
There are two cases: \\
{\it Case I:} $\phi\equiv 0$. In this case \dref{eig-equ} becomes
\vspace{-2mm}
$$ (S_r+l_rq_r^\top)X_r \m=\m \l X_r,\qquad (S_d+l_dn(1)^\top)
   X_d \m=\m \l X_d,$$
which has nontrivial solutions $(\l_{rj},[X_{rj},0_{n_d\times1}])$,
$j=1,2,\ldots n_r$ and $(\l_{dj},$ $[0_{n_r\times1},X_{dj}])$, $j=1,
2,\ldots n_d$. Hence, $(\l_{rj},F_{1j})=(\l_{rj},[X_{rj},0_{n_d
\times1},0])$, $j=1,2,\ldots n_r$, together with $(\l_{dj},F_{1(j+
n_r)})=(\l_{dj},[0_{n_r\times1},X_{dj},0])$, $j=1,2,\ldots
n_d$ are eigen-pairs of $\AAA$.

\noindent {\it Case II:} $\phi\neq 0$. Now \vspace{-1mm}
$$ \m\ \ \ \ \ -i\phi''(x)-c_o\phi(x) \m=\m \l\phi(x),\qquad \phi'(0)
   \m=\m 0,\qquad \phi'(1) \m=\m 0 \m,$$
which has nontrivial solutions $(\l_j,\phi_j(x))$: \vspace{-1mm}
$$ \l_j \m=\m j^2\pi^2i-c_o,\qquad \phi_j(x)=\cos(j\pi x),\ \ j=0,1,2,
   \ldots\ .\vspace{-1mm}$$
Substituting $(\l_j,\phi_j(x))$ into the first and the second equation
of \dref{eig-equ}, we get \vspace{-2mm}
$$ X_{r}^j \m=\m 0_{n_r\times1},\qquad X_{d}^j \m=\m -[(S_d+l_dn(1)
   ^\top) - (j^2\pi^2i-c_o)I_{n_d\times n_d}]^{-1}l_d\cos(j\pi)
   \m. \vspace{-1mm}$$
Thus we have found for $\AAA$ the eigen-pairs $(\l_j,F_{2j})$, for
$j=0,1,2,\ldots$\m, where \vspace{-2mm}
$$ F_{2j} \m=\m [0_{n_r\times 1},-[(S_d+l_dn(1)^\top)-(j^2\pi^2
   i-c_o)I_{n_d\times n_d}]^{-1}l_d\cos(j\pi),\m\cos(j\pi\cdot)].$$

Now we prove that the set $\{F_{1j_1}(x),F_{2j_2}(x)\m|\ j_1=1,2,
\ldots n_w,\ j_2=0,1,2, \ldots\}$ is a Riesz basis for $\rline^{n_w}
\times L^2[0,1]$. Indeed, let us denote by $G_j$ the first part of
$F_{1j}$, so that $G_j\in \rline^{n_w}$ and $F_{1j}=[G_j,0]$. Since
the set $\{G_j\m|\ j=1,2, \ldots n_w\}$ and the set $\{\cos(j\pi\cdot)
\m|\ j=0,1,2,\ldots\}$ form Riesz bases for $\rline^{n_w}$ and for
$L^2[0,1]$, respectively, \m $\{F_{1j}\m|\ j=1, 2,\ldots n_w\}\cup
\{F_{2j}^*=[0_{n_w\times 1}, \cos(j\pi\cdot)]\m|\ j=1,2,\ldots\}$ is a
Riesz basis in $\rline^{n_w}\times\hline$. Moreover, the set that we
want to prove to be a Riesz basis is quadratically close to the Riesz
basis that we have just found: \vspace{-2mm}
\begin{equation} \label{F-err}
\sum_{j=0}^{\infty} \|F_{2j}-F_{2j}^*
   \|^2_{\rline^{n_w}\times\hline} \hspace{40mm} \m \vspace{-2mm}
\end{equation}
$$ \begin{array}{l}\disp 
   \m=\m \sum_{j=0}^{\infty}
   \|[(S_d+l_dn(1)^{\top})-(j^2\pi^2i-c_o)I_{n_d\times n_d}]^{-1}l_d
   \|^2_{\rline^{n_d}}\crr\disp =\m \sum_{j=0}^{\infty}\frac{1}
   {|j^2\pi^2i-c_o|^2}\|[(S_d+l_dn(1)^{\top})/(j^2\pi^2i-c_o)-I_{n_d
   \times n_d}]^{-1}l_d\|^2_{\rline^{n_d}}. \end{array}$$

Since \vspace{-2mm}
$$ \lim_{j\to\infty}\| [(S_d+l_d n(1)^{\top})/(j^2\pi^2i-c_o)-
   I_{n_d\times n_d}]^{-1}l_d \|^2_{\rline^{n_d}} \m=\m \|l_d\|
   ^2_{\rline^{n_d}},$$
it follows from \dref{F-err} that \m $\sum_{j=0}^\infty\|F_{2j}-F_{2j}
^*\|^2_{\rline^{n_w}\times\hline}<\infty$. By the classical theorem of
Bari, $\{F_{1j}\}_{j=1}^{n_w}\cup \{F_{2j}\}_{j=0}^{+\infty}$ forms a
Riesz basis for $\rline^{n_w} \times\hline$. This shows that $\AAA$
generates an operator semigroup on $\rline^n\times\hline$, for which
the spectrum determined growth assumption holds. As a consequence, the
system \dref{te-equ} admits a unique solution. Since $\sup\{\Re\l\m|\
\l\in\sigma(\AAA)\}<0$, $e^{\AAA t}$ is an exponentially stable
operator semigroup, which, together with the boundedness of the
transformations \dref{trans-tz-u} and \dref{tu-u-trans-nn}, implies
\dref{Thm-re-err-expon}.
\end{proof}

\section{Output feedback regulation} 

\ \ \ By Theorem \ref{Thm-obser} we have obtained the estimated states
$\widehat{w}$ and $\widehat{z}(x,t)$ for $w$ and $z(x,t)$,
respectively. Since the state feedback control \dref{state-con}
achieves the output regulation, we naturally propose the following
output feedback control law: \vspace{-2mm}
\begin{equation} \label{state-con-out}
   u(t) \m=\m k(1,1) \widehat{z}(1,t) + \int_0^1 k_x(1,\xi)
   \widehat{z}(\xi,t)\dd\xi + m_w^\top \widehat{w}(t). \vspace{-1mm}
\end{equation}
Here we can see that the terms $k(1,1)\widehat{z}(1,t)+\int_0^1 k_x
(1,\xi)\widehat{z}(\xi,t)\dd\xi$ are to stabilize the system
\dref{Sch} and the term $m_w^\top\widehat{w}(t)$ is to track the
reference signal $r(t)=p_r^\top w(t)$. Now we turn to the closed-loop
system composed of \dref{Sch}, \dref{dis-produce}, \dref{wr-obser},
\dref{Sch-obser-r} and \dref{state-con-out}, that is \vspace{-2mm}
\begin{equation} \label{Sch-output-closed}
   \left\{\begin{array}{l} z_t(x,t) \m=\m -iz_{xx}(x,t)+h(x)z(x,t)+
   g(x)d_1(t),\crr\disp z_x(0,t) \m=\m  -iqz(0,t)+d_2(t),\crr\disp
   z_x(1,t) \m=\m k(1,1)\widehat{z}(1,t) + \int_0^1 k_x(1,\xi)
   \widehat{z}(\xi,t)\dd\xi+m_w^\top\widehat{w}(t),\end{array}\right.
\end{equation}
\begin{equation} \label{Sch_bis}
   \left\{\begin{array}{l} \dot{w}(t) \m=\m Sw(t),\crr\disp
   \dot{\widehat{w}}_r(t) \m=\m S_r \widehat{w}_r(t)+l_r(q_r^\top
   \widehat{w}_r(t)-r(t)),\crr\disp \dot{\widehat{w}}_d(t) \m=\m
   S_d\widehat{w}_d(t)+l_d(\widehat{z}(1,t)-y_m(t)),\crr\disp
   \widehat{z}_t(x,t) \m=\m -i\widehat{z}_{xx}(x,t)+h(x)\widehat{z}
   (x,t)+g(x)q^{\top}_{d_1}\widehat{w}_d(t)\crr\disp\ 
   \ \ \ \ \ \ \ \ \ \ \ \ \ \ +l(x)\left[\widehat{z}
   (1,t)-y_m(t)\right],\crr\disp \widehat{z}_x(0,t) \m=\m -iq
   \widehat{z}(0,t) + q^\top_{d_2}\widehat{w}_d(t),\crr\disp
   \widehat{z}_x(1,t) \m=\m k(1,1)\widehat{z}(1,t)+\int_0^1 k_x(1,
   \xi)\widehat{z}(\xi,t)\dd\xi + m_w^{\top}\widehat{w}(t)\crr\disp\ 
   \ \ \ \ \ \ \ \ \ \ \ \ \ \ +l_0\left[
   \widehat{z}(1,t)-y_m(t)\right],\end{array}\right.
\end{equation}
where the gains $l(\cdot)$, $l_0$ are given by \dref{lxl0-expr-fin}
and the gain $l_r$ is chosen so that $S_r+l_rq_r^\top$ is Hurwitz.
The following is the main result of this section.

\begin{theorem} \label{Thm-outre}
Suppose that the conditions in Theorems {\rm\ref{Thm-statere}} and
{\rm\ref{Thm-obser}} hold.

Then for any initial state $(z_0(x),w(0),\widehat{z}_0(x),\widehat{w}
(0))\in\hline\times\rline^{n_w}\times\hline\times\rline^{n_w}$, the
closed-loop system {\rm\dref{Sch-output-closed}-\dref{Sch_bis}} admits
a unique solution $(z(\cdot,t),w(t),\widehat{z}(\cdot,t),$ $\widehat{w}
(t))\in C([0,\infty);\hline\times\rline^{n_w}\times\hline\times\rline
^{n_w})$. Moreover, there exist $M\geq 1$, $\mu>0$ such that
$$ \begin{array}{l}\disp
\|(\widehat{w}_r(t)-w_r(t),\widehat{w}_d(t)-w_d(t),\widehat{z}
   (\cdot,t)-z(\cdot,t))\| \m\crr\disp
   \leq Me^{-\mu t}\|(\widehat{w}_r(0)-
   w_r(0),\widehat{w}_d(0)-w_d(0),\widehat{z}_0-z_0)\|.\end{array}$$

The observer based controller (with internal loop) {\rm
\dref{wr-obser}, \dref{Sch-obser-r}} and {\rm\dref{state-con-out}}
solves the output feedback regulator problem for the plant {\rm
\dref{Sch}} with the exosystem {\rm\dref{dis-produce}}. This means
that the output error $e_y(t)=y(t)-r(t)=C_{e\L}[z(\cdot,t)]-p_r^\top
w(t)$ for the closed-loop system {\rm\dref{Sch-output-closed}-%
\dref{Sch_bis}} satisfies $e_y\in L^2_\alpha[0,\infty)$ for some
$\alpha<0$. If $C_e$ is bounded, then there exist $m_0,\mu_0>0$ ($m_0$
depends on the initial state mentioned above) such that we have
$|e_y(t)|\leq m_0 e^{-\mu_0 t}$ for all $t\geq 0$.
\end{theorem}

\begin{proof}
Using the error variables $\widetilde{w}$ and $\widetilde{z}$ defined
before \dref{Sch-obser-err-z}, we can write an equivalent system to
\dref{Sch-output-closed}-\dref{Sch_bis} as follows: \vspace{-1mm}
\begin{equation} \label{Sch-output-closed-equiv}
   \left\{\begin{array}{l} z_t(x,t) \m=\m -iz_{xx}(x,t)+h(x)z(x,t)+
   g(x)d_1(t),\crr\disp z_x(0,t) \m=\m -iqz(0,t)+d_2(t),\crr\disp
   z_x(1,t) \m=\m k(1,1)[z(1,t)+\widetilde{z}(1,t)]+ m_w^\top[w(t)+
   \widetilde{w}(t)]\crr\disp\ \ \ \ \ \ \ \ \ \ \ \ \
   +\int_0^1 k_x(1,\xi)[z(\xi,t)+\widetilde{z}(\xi,t)]\dd\xi,\crr\disp
   \dot{w}(t) \m=\m Sw(t),\qquad w(0) \m=\m w_0\in\rline^{n_w}, 
   \end{array}\right.
\end{equation}
\begin{equation} \label{Sch-output-closed-equiv_bis}
   \left\{\begin{array}{l} \dot{\widetilde{w}}_r(t) \m=\m (S_r+l_r
   q_r^\top)\widetilde{w}_r(t),\qquad \dot{\widetilde{w}}_d \m=\m S_d
   \widetilde{w}_d(t)+l_d\widetilde{z}(1,t),\crr\disp \widetilde
   {z}_t(x,t) \m=\m -i\widetilde{z}_{xx}(x,t)+h(x)\widetilde{z}(x,t)+
   g(x) q^\top_{d_1}\widetilde{w}_d(t)+l(x)\widetilde{z}(1,t),\crr
   \disp \widetilde{z}_x(0,t) \m=\m -iq\widetilde{z}(0,t)+q^{\top}
   _{d_2} \widetilde{w}_d(t),\qquad \widetilde{z}_x(1,t) \m=\m l_0
   \widetilde{z}(1,t). \end{array}\right.
\end{equation}
The ``$(\widetilde{w}_r,\widetilde{w}_d,\widetilde{z})$-part'' in
\dref{Sch-output-closed-equiv_bis} has been shown to be exponentially
stable in Theorem \ref{Thm-obser}. Now we only need to consider the
``$(w,z)$-part'' in \dref{Sch-output-closed-equiv}, which we rewrite
as
\begin{equation} \label{Sch-output-closed-equiv-wz}
   \left\{\begin{array}{l} z_t(x,t) \m=\m -iz_{xx}(x,t)+h(x)z(x,t)+
   g(x)d_1(t),\crr \disp z_x(0,t) \m=\m -iqz(0,t)+d_2(t),\crr\disp
   z_x(1,t) \m=\m k(1,1)[z(1,t)+\widetilde{z}(1,t)] + m_w^\top[w(t)+
   \widetilde{w}(t)]\crr\disp\ \ \ \ \ \ \ \ \ \ \ \ \ \ \ \
   + \int_0^1 k_x(1,\xi)[z(\xi,t) + \widetilde{z}(\xi,t)]\dd\xi,\crr
   \disp \dot{w}(t) \m=\m Sw(t),\qquad w(0)
   \m=\m w_0\in\rline^{n_w}.\end{array}\right.
\end{equation}
Under the backstepping transformation \dref{trans-z-u}, the
``$z$-part'' of system \dref{Sch-output-closed-equiv-wz} can be
converted into the following equivalent system:
$$ \left\{\begin{array}{l} v_t(x,t) \m=\m -iv_{xx}(x,t)-c_sv(x,t)+
   \Fscr[g](x)d_1(t)-k(x,0)d_2(t),\crr\disp v_x(0,t) \m=\m d_2(t),
   \crr\disp v_x(1,t) \m=\m k(1,1)\widetilde{z}(1,t) + \int_0^1 k_x
   (1,\xi)\widetilde{z}(\xi,t)\dd\xi + m_w^\top[w(t)+\widetilde{w}
   (t)]. \end{array}\right.$$
Further, by the transformation \dref{tu-u-trans}, the above system
is equivalent to
\begin{equation} \label{ORSch-tu-closed}
   \left\{\begin{array}{l}\disp \widetilde{v}_{t}(x,t) \m=\m -i
   \widetilde{v}_{xx}(x,t)-c_s\widetilde{v}(x,t),\crr\disp
   \widetilde{v}_x(0,t)=0,\crr\disp \widetilde{v}_x(1,t) \m=\m
   k(1,1)\widetilde{z}(1,t)+\int_0^1 k_x(1,\xi)\widetilde{z}(\xi,t)
   \dd\xi + m_w^\top\widetilde{w}(t).\end{array}\right.
\end{equation}
Note that this is different from the system \dref{ORSch-tu}, that was
derived for the case of state feedback.

Now we show that $\|\widetilde{v}(\cdot,t)\|\leq M_0e^{-\mu_0t}$ for
some $M_0,\mu_0>0$. To do this, first we show that $\widetilde{e}(1,
\cdot)$ in \dref{te-equ} belongs to $L^2_{-\alpha_o}[0,\infty)$ for
some $\alpha_o\in(0,c_o/2)$, where $L^2_{-\alpha_o}[0,\infty)$ is
defined after \dref{ourerr-def}. Define the sequence $(c_j)_{j\in
\nline}$ by $c_j=\cos(j\pi)$. Obviously, this sequence satisfies the
Carleson measure criterion, see \cite[Definition 5.3.1]{obs_book}.
Define the observation operator $\mathcal{C}z=\sum_{j=0}^{\infty}c_j
z_j$, where $z=\sum_{j=0}^{\infty}z_j\cos(j\pi x)$ with $(z_j)_{j=0}
^\infty\in l^2$. From the proof of Theorem \ref{Thm-obser} (Case II)
we have that (i) system \dref{te-equ} is associated with a diagonal
group $\tline$ with $(\tline_t z)_j=z_je^{(j\pi^2i-c_o)t}\ (\forall
j\in\nline)$ on $l^2$; (ii) the generator $A_0$ of the diagonal
group $\tline$ satisfies $A_0\cos(j\pi x)=\l_j\cos(j\pi x)$ with
$\l_j=j\pi^2i-c_o$; (iii) $\widetilde{e}(1,t)=\mathcal{C}\widetilde
{e}(x,t)$. Moreover, it is easy to verify that \vspace{-2mm}
$$ \sum_{{\rm Im}\l_j\in[n,n+1)}|c_j|^2 \m\leq\m 1 \FORALL
   n\in\zline \m.$$
It follows from \cite[Proposition 5.3.5.]{obs_book} that
$\mathcal{C}$ is an admissible observation operator for $\tline$.
With \cite[Proposition 4.3.6]{obs_book} we get that $\widetilde{e}
(1,\cdot)\in L^2_{-\alpha_o}[0,\infty)$ for some $\alpha_o\in(0,
c_o/2)$: \vspace{-2mm}
\begin{equation} \label{eLt-est}
   \int_0^\infty |e^{\alpha_os}\widetilde{e}(1,s)|^2 \dd s :=\m C_1
   \m<\m \infty \m.
\end{equation}
By \dref{trans-tz-u} and \dref{tu-u-trans-nn}, we get
$\widetilde{z}(1,t)=\widetilde{e}(1,t)+n(1)^{\top}\widetilde{w}_d(t)$.
From Theorem \ref{Thm-obser} and \dref{eLt-est} we know that
\vspace{-2mm}
$$ k(1,1)\widetilde{z}(1,t) + \int_0^1 k_x(1,\xi)\widetilde{z}(\xi,t)
   \dd\xi + m_w^{\top}\widetilde{w}(t) :=\m \eta_1(t)+\eta_2(t).$$
with \vspace{-2mm}
$$ \begin{array}{l} \eta_1(t)=k(1,1)\widetilde{e}(1,t), \ \ \ \ \crr
   \disp \eta_2(t)=k(1,1)n(1)^{\top}\widetilde{w}_d(t)+\int_0^1 k_x
   (1,\xi)\widetilde{z}(\xi,t)\dd\xi + m_w^{\top}\widetilde{w}(t),
   \end{array}$$
which satisfies 
\begin{equation} \label{eta1eta2}
   \m\ \ \int_0^\infty |e^{\alpha_os}\eta_1(s)|^2 \dd s:=k^2(1,1)C_1 <
   \infty,\ \ \ |\eta_2(t)|\leq M_1e^{-\mu_1t} \ \ \forall t\geq 0,
\end{equation}
for some $M_1,\mu_2>0$. Using the operator $\aline$ from \dref{Alina},
we write the system \dref{ORSch-tu-closed} as \vspace{-2mm}
$$ \frac{\dd}{\dd t}\widetilde{v}(\cdot,t) \m=\m (\aline-c_s I)
   \widetilde{v}(\cdot,t) + B\left[\eta_1(t)+\eta_2(t)\right],$$
where $B=i\delta(\cdot-1)$. Clearly $e^{(\aline -c_s I)t}$ is
exponentially stable. As in the proof of Proposition \ref{Brighitte},
we have that $B$ is an admissible control operator for $e^{(\aline
-c_s I)t}$. Thus, it follows from \cite[Proposition 4.2.5]{obs_book}
that the solution $\widetilde{v}$ is a continuous $L^2[0,1]$-valued
function of $t$ given by \vspace{-2mm}
\begin{equation} \label{tv-expr}
   \widetilde{v}(\cdot,t) \m=\m e^{(\aline-c_s I)t}\widetilde{v}
   (\cdot,0) + \int_0^t e^{(\aline-c_s I)(t-s)}B[\eta_1(s)+\eta_2(s)]
   \dd s \m.
\end{equation}
Moreover, from the exponential stability of $e^{(\aline-c_s I)t}$ and
\cite[Lemma 2.1]{ZhouWeiss_SIAM}, we have that $\|\widetilde{v}
(\cdot,t)\|\leq M_0e^{-\mu_0t}$ for some $M_0,\mu_0>0$. Noting the
formula \dref{sol-me} for $e_y$ and Proposition \ref{utolso}, the
admissibility of the observation operator $C_e$ implies that the
tracking error system is exponentially stable in the sense that $e_y=
C_{e\L}\Fscr^{-1} [\widetilde{v}]\in L^2_{\alpha}[0,\infty)$ with
$\alpha\in(-\mu_1,0)$.  In particular, if the observation operator
$C_e$ is bounded, then by the boundedness of the transformation
$\Fscr^{-1}$ there exists a constant $C_2>0$ such that \m $|e_y(t)|=
|C_e\Fscr^{-1}[\widetilde{v}] (\cdot,t)|\leq C_2\|\widetilde{v}
(\cdot,t)\|\leq C_2 M_0 e^{-\mu_0 t}$, so that \dref{ey-to0} holds.

The inequality in this theorem follows from Theorem \ref{Thm-obser}.
By \dref{trans-z-u} and \dref{tu-u-trans}, we have \vspace{-2mm}
$$z(x,t) \m=\m \Fscr^{-1}[\widetilde{v}+mw](x,t).$$
Since $\lim_{t\to\infty}\|\widetilde{v}(\cdot,t)\|=0$, $w(t)$ is
bounded for all $t\geq0$ and the transformation $\mathcal{F}^{-1}$ is
bounded, we know that $\|z(\cdot,t)\|$ is bounded for all $t\geq0$.
It follows from the inequality in this theorem that all internal
signals $z(\cdot,t),w(t),\widehat{z}(\cdot,t),\widehat{w}(t)$ are
bounded.
\end{proof}

\begin{remark} {\rm
A very concise version of this paper, with weaker results and missing
proofs, was presented at a conference \cite{ZhouWeiss_IFAC:2017}.}
\end{remark}

\noindent
{\bf Acknowledgment.} The second author is grateful to Gail Weiss
(his daughter) for help in the proof of Proposition \ref{Gail}.


\end{document}